\documentclass[]{interact}

\usepackage{epstopdf}% To incorporate .eps illustrations using PDFLaTeX, etc.
\usepackage[caption=false]{subfig}% Support for small, `sub' figures and tables
\usepackage{hyperref}
\usepackage[numbers,sort&compress]{natbib}% Citation support using natbib.sty
\bibpunct[, ]{[}{]}{,}{n}{,}{,}% Citation support using natbib.sty
\makeatletter% @ becomes a letter
\def\NAT@def@citea{\def\@citea{\NAT@separator}}% Suppress spaces between citations using natbib.sty
\makeatother% @ becomes a symbol again

\theoremstyle{plain}% Theorem-like structures provided by amsthm.sty
\newtheorem{theorem}{Theorem}[section]
\newtheorem{lemma}[theorem]{Lemma}
\newtheorem{corollary}[theorem]{Corollary}
\newtheorem{proposition}[theorem]{Proposition}

\theoremstyle{definition}
\newtheorem{definition}[theorem]{Definition}

\theoremstyle{remark}
\newtheorem{remark}{Remark}

\newcommand{\be}{\begin{equation}}
\newcommand{\ee}{\end{equation}}

\newcommand{\bea}{\begin{eqnarray}}
\newcommand{\eea}{\end{eqnarray}}

\begin{document}

\title{Generalized $q$-Bernoulli polynomials generated by Jackson $q$-Bessel functions}

\author{
\name{S.Z.H. Eweis\textsuperscript{a}\thanks{CONTACT S.Z.H. Eweis. Email: Sahar.Zareef@bsu.edu.eg} and Z.S.I.  Mansour\textsuperscript{b}\thanks{CONTACT Z. S. Mansour. Email: zsmansour@cu.edu.eg}}
\affil{\textsuperscript{a}Mathematics and Computer Science Department,  Faculty of Science, Beni-Suef University, Beni-Suef, Egypt\\
 \textsuperscript{b}Department of Mathematics, Faculty of Science, Cairo University, Giza, Egypt.}
}

\maketitle

\begin{abstract}
 In this paper, we introduce the polynomials $B^{(k)}_{n,\alpha}(x;q)$ generated by a function including Jackson $q$-Bessel functions $J^{(k)}_{\alpha}(x;q)$ $ (k=1,2,3),\,\alpha>-1$. The cases $\alpha=\pm\frac{1}{2}$ are  the $q$-analogs of Bernoulli and Euler$^{,}$s polynomials introduced by Ismail and Mansour for $(k=1,2)$,  Mansour and Al-Towalib for $(k=3)$. We study the main properties of these polynomials, their large $n$ degree asymptotics and give their connection coefficients with the $q$-Laguerre polynomials and little $q$-Legendre polynomials.
\end{abstract}

\begin{keywords}
$q$-Bessel functions,  $q$-Bernoulli polynomials and numbers, asymptotic expansions, cauchy residue theorem.
\end{keywords}
\begin{amscode}
05A30, 11B68, 30E15, 32A27
\end{amscode}

\section{Introduction and Preliminaries}\label{Sec.1}

{\small The Bernoulli polynomials $ \left(B_{n}(x)\right)_{n}$ are defined by the generating function
$$\frac{te^{xt}}{e^{t}-1}=\sum_{n=0}^{\infty}B_{n}(x)\frac{t^{n}}{n!},\quad |t|<2\pi.$$
 In a series of papers, Frappier \cite{Frappier1, Frappier2, F3} studied  the generalized  Bernoulli polynomials $ B_{n,\alpha}(x)$, defined by the generating function
 \begin{equation}\label{q9321}
   \frac{e^{(x-\frac{1}{2})t}}{g_{\alpha}(\frac{it}{2})}=\sum_{n=0}^{\infty}B_{n,\alpha}(x)\frac{t^{n}}{n!},\quad  |t|<2 j_{1,\alpha},
 \end{equation}
where
$$ g_{\alpha}(t)=2^{\alpha}\Gamma(\alpha+1)\frac{ J_{\alpha}(t)}{t^{\alpha}},$$
$ J_{\alpha}(t)$ is the Bessel function of the first kind of order $\alpha,$ and $j_{1,\alpha} $ is the smallest positive zero of $J_{\alpha}(t)$.}
{\small Ismail and Mansour, see \cite{IZ},  introduced $q$-pair of analogs of the Bernoulli polynomials by the generating functions
 \begin{align}\label{q8}
 \begin{split}
  \frac{te_{q}(xt)}{e_{q}(\frac{t}{2})E_{q}(\frac{t}{2})-1}&=\sum_{n=0}^{\infty}b_{n}(x;q)\frac{t^{n}}{[n]_{q}!},\\
  \frac{t E_{q}(xt)}{e_{q}(\frac{t}{2})E_{q}(\frac{t}{2})-1}&=\sum_{n=0}^{\infty}B_{n}(x;q)\frac{t^{n}}{[n]_{q}!}.\end{split}
\end{align}
They also defined a pair of $q$-analogs of the Euler polynomials by the generating functions
\begin{align}\label{q13248}
 \begin{split}
  \frac{2e_{q}(xt)}{E_{q}(\frac{t}{2})e_{q}(\frac{t}{2})+1}&=\sum_{n=0}^{\infty}e_{n}(x;q)\frac{t^{n}}{[n]_{q}!},\\
  \frac{2 E_{q}(xt)}{E_{q}(\frac{t}{2})e_{q}(\frac{t}{2})+1}&=\sum_{n=0}^{\infty}E_{n}(x;q)\frac{t^{n}}{[n]_{q}!},\end{split}
\end{align}
where
 \[[n]_{q}!=\frac{(q;q)_{n}}{(1-q)^{n}}\quad (n\in\mathbb{ N}), \quad (a;q)_{n}=\left\{
                                                                           \begin{array}{ll}
                                                                             1, & \hbox{$n=0$;} \\
                                                                              \displaystyle\prod_{k=0}^{n-1}(1-aq^{k}), & \hbox{$n\in\mathbb{N}$,}
                                                                           \end{array}
                                                                         \right.
 \]
and $a\in\mathbb{C}$, see \cite{Gasper}.  The functions $E_q(x)$ and $e_q(x)$ are the $q$-analogs of the exponential functions defined by
{\small \begin{align}\label{h8}
 \begin{split}
 E_q(x):= (-x(1-q);q)_{\infty}&=\sum_{n=0}^{\infty}\frac{q^{\frac{n(n-1)}{2}}(1-q)^{n}x^{n}}{(q;q)_{n}}, \,\, x\in \mathbb{C}, \\
  e_q(x):=\frac{1}{(x(1-q);q)_{\infty}}&=\sum_{n=0}^{\infty}\frac{(1-q)^{n}x^{n}}{(q;q)_{n}}, \,\, |x|< \frac{1}{1-q}, \end{split}
  \end{align}}
  see e.g. \cite{Gasper}.
}
{\small In \cite{marim},  Mansour and Al-Towalib introduced  $q$-analogs of Bernoulli and Euler polynomials by the generating functions
\begin{equation}\label{q439}
\begin{split}
  &\frac{t \exp_{q}(xt)\exp_{q}(\frac{-t}{2})}{\exp_{q}(\frac{t}{2})-\exp_{q}(\frac{-t}{2})}=\sum_{n=0}^{\infty}\tilde{B}_{n}(x;q)\frac{t^{n}}{[n]_{q}!},\\&
   \frac{2 \exp_{q}(xt)\exp_{q}(\frac{-t}{2})}{\exp_{q}(\frac{t}{2})+\exp_{q}(\frac{-t}{2})}=\sum_{n=0}^{\infty}\tilde{E}_{n}(x;q)\frac{t^{n}}{[n]_{q}!},
   \end{split}
\end{equation}}
{\small where
 \begin{equation*}\label{q4839}
\exp_{q}(x)=\sum_{n=0}^{\infty}q^{\frac{n(n-1)}{4}}\frac{x^{n}}{[n]_{q}!},\quad x\in \mathbb{C},
\end{equation*}
is a $q$-analog of the exponential function. This $q$-exponential function has the property
$\lim_{q\rightarrow1}\exp_{q}(x) = e^{x}$ for $x \in \mathbb{C}$. It is an
entire function of $x$ of order zero, see \cite[Eq. (1.3.27), P. 12]{Gasper}.}
\vskip0.2cm

% ----------------------------------------------------------------

  {\small  In this paper,  we use $\mathbb{N} $  to denote the set of positive integers and  $\mathbb{N}_{0} $  to denote the set of non-negative integers. Throughout this paper, unless otherwise is stated, $q$ is a positive number that is less than one. We follow Gasper and Rahman \cite{Gasper} to define the
$q$-shifted factorial, the $q$-binomial coefficients, and the $q$-gamma function.  The $q$-integer number $[n]_{q}$ is defined by
\begin{equation*}
[n]_{q}=\frac{1-q^n}{1-q},\quad n\in\mathbb{ N}_0.
\end{equation*}

Jackson in \cite{Jackson} defined the $q$-difference operator by
         \begin{equation*}
           D_q f(z)= \dfrac{f(qz)-f(z)}{z(q-1)}, \quad z\neq 0.
         \end{equation*}
}
{\small
The symmetric $q$-difference operator is defined by, see \cite{Cardoss, Gasper},
\begin{equation*}
\delta_{q,z} f(z)= \frac{{}f(q^{\frac{1}{2}}z)-f(q^{\frac{-1}{2}}z)}{(q^{\frac{1}{2}}-q^{\frac{-1}{2}})z}, \quad z\neq 0.
\end{equation*}

}
{\small
The $q$-trigonometric functions $\sin_{q}z,\, \cos_{q}z,\,Sin_{q}z$ and $Cos_{q} z$ are defined by
 \begin{equation*}
 \begin{split}
   &\sin_{q}z=\dfrac{e_{q}(iz)-e_{q}(-iz)}{2i},\quad \cos_{q}z=\dfrac{e_{q}(iz)+e_{q}(-iz)}{2},\,\,|z|<1,\\&
   Sin_{q}z=\dfrac{E_{q}(iz)-E_{q}(-iz)}{2i},\quad Cos_{q}z=\dfrac{E_{q}(iz)+E_{q}(-iz)}{2},\,\,z\in\mathbb{C},
   \end{split}
 \end{equation*}}
 see  \cite{Gasper, Annaby}.
{\small The $q$-sine and cosine functions $S_{q}(z),\,\, C_{q}(z)$ are defined by the $q$-Euler$^{,}$s formula
\begin{equation*}
  \exp_{q}(iz):=C_{q}(z)+iS_{q}(z),
\end{equation*}
where}
{\small\begin{equation*}
\begin{split}
  C_{q}(z)=\sum_{n=0}^{\infty}(-1)^{n}\frac{q^{n(n-\frac{1}{2})}}{[2n]_{q}!}z^{2n},\quad  S_{q}(z)=\sum_{n=0}^{\infty}(-1)^{n}\frac{q^{n(n+\frac{1}{2})}}{[2n+1]_{q}!}z^{2n+1},
  \end{split}
\end{equation*}}
cf. \cite[P. 2]{Cardoss}. {\small The hyperbolic functions $Sh_{q}(z)$ and $Ch_{q}(z)$ are defined for ${\small z\in\mathbb{C}}$ by}
{\small\begin{equation}\label{f40505}
\begin{split}
  &Sh_{q}(z):=-iS_{q}(iz)=\frac{\exp_{q}(z)-\exp_{q}(-z)}{2},\\&
  Ch_{q}(z):= C_{q}(iz)=\frac{\exp_{q}(z)+\exp_{q}(-z)}{2}.
  \end{split}
\end{equation}}

{\small
There are three known $ q$-analogs of  the Bessel function that are due to Jackson \cite{Jackson}. These
are denoted by $ J^{(k)}_{\alpha}(t;q)\, (k=1,2,3)$ and defined by
\begin{equation*}
    J_{\alpha}^{(1)}(t;q)=\frac{(q^{\alpha+1};q)_{\infty}}{(q;q)_{\infty}}\sum_{n=0}^{\infty}(-1)^n\frac{(\frac{t}{2})^{2n+\alpha}}{(q;q)_{n}(q^{\alpha+1};q)_{n}}
    \quad(|t|<2),
 \end{equation*}
\begin{equation*}
    J_{\alpha}^{(2)}(t;q)=\frac{(q^{\alpha+1};q)_{\infty}}{(q;q)_{\infty}}\sum_{n=0}^{\infty}(-1)^n\frac{q^{n(\alpha+n)}(\frac{t}{2})^{2n+\alpha}}{(q;q)_{n}(q^{\alpha+1};q)_{n}}\quad(t\in \mathbb{C} ),
 \end{equation*}
\begin{equation*}
    J_{\alpha}^{(3)}(t;q)=\frac{(q^{\alpha+1};q)_{\infty}}{(q;q)_{\infty}}\sum_{n=0}^{\infty}(-1)^n\frac{q^{\frac {n(n+1)}{2}}t^{2n+\alpha}}{(q;q)_{n}(q^{\alpha+1};q)_{n}}
    \quad(t\in \mathbb{C}).
 \end{equation*}}

{\small For convenience, we set
{\small\begin{equation}\label{g7907}
\begin{split}
& \mathcal{J}_{\alpha}^{(k)} (t;q):=\frac{(q;q)_{\infty}}{(q^{\alpha+1};q)_{\infty}}(\frac {t}{2})^{-\alpha} J_{\alpha}^{(k)} (t;q)\quad( k=1,2),\\&
  \mathcal{J}_{\alpha}^{(3)} (t;q):=\frac{(q;q)_{\infty}}{(q^{\alpha+1};q)_{\infty}}t^{-\alpha} J_{\alpha}^{(3)}(t;q).\end{split}
\end{equation}}
 The functions $\mathcal{J}_{\alpha}^{(k)} (t;q)\, (k=1,2,3)$  are called the modified Jackson $q$-Bessel functions.} {\small From now on, we use $(j^{(k)}_{m,\alpha})_{m=1}^{\infty}$ to denote the positive zeros of $J^{(k)}_{\alpha}(\cdots;q^{2})$ arranged in increasing order of  magnitude. Consequently, $j^{(k)}_{1,\alpha}$ is the smallest positive zero of $J^{(k)}_{\alpha}(\cdots;q^{2})\,( k=1,2,3)$.}
\vskip0.4cm
{\small
This paper is organized as follows. In Section 2, we introduce three $q$-analogs of the generalized Bernoulli polynomials defined in (\ref{q9321}).  The generating functions of these $q$-analogs include the three $q$-analogs of Jackson $q$-Bessel functions mentioned above. We also include the main properties of these $q$-analogs.  Section 3  introduces a $q$-Fourier expansion for the generalized Bernoulli numbers related to the first and second Jackson $q$-Bessel functions. Also,  their large $n$ degree asymptotics is derived. Finally, in Section 4 as an application, we introduce the connection coefficients between $q$-analogs and certain $q$-orthogonal polynomials.}
% ----------------------------------------------------------------
\section{ Generalized $q$-Bernoulli polynomials generated by Jackson $q$-Bessel functions} \label{$q$-Bernoulli polynomials generated by  Jackson $q$-Bessel functions}
{\small This section introduces three $ q$-analogs of the generalized Bernoulli polynomials introduced by Frappier in \cite{Frappier1, Frappier2, F3}.

\begin{definition}\label{q-Bernoulli}{\small The generalized $q$-Bernoulli polynomials $B^{(k)}_{n,\alpha}(x;q)\,( k=1,2,3)$ are \\defined by  the generating functions
\begin{equation}\label{q14}
  \quad \frac{e_{q}(xt)e_{q}(\frac{-t}{2})}{g^{(1)}_{\alpha}(it;q)}=\sum_{n=0}^{\infty}B^{(1)}_{n,\alpha}(x;q)\frac{t^{n}}{[n]_{q}!},\quad|t|<
 \frac{ j^{(1)}_{1,\alpha}}{1-q},
\end{equation}
\begin{equation}\label{q13}
  \quad \frac{E_{q}(xt)E_{q}(\frac{-t}{2})}{g^{(2)}_{\alpha}(it;q)}=\sum_{n=0}^{\infty}B^{(2)}_{n,\alpha}(x;q)\frac{t^{n}}{[n]_{q}!},\quad|t|<\frac{
  j^{(2)}_{1,\alpha}}{1-q},
\end{equation}
\begin{equation}\label{q197}
   \frac{\exp_{q}(xt)\exp_{q}(\frac{-t}{2})}{g^{(3)}_{\alpha}(it;q)}=\sum_{n=0}^{\infty}B^{(3)}_{n,\alpha}(x;q)\frac{t^{n}}{[n]_{q}!},
   \quad|t|<\frac{2q^{\frac{1}{4}}j^{(3)}_{1,\alpha}}{1-q},
\end{equation}
where  $g^{(k)}_{\alpha}(t;q) \,( k=1,2,3 )$  are the functions defined  for $(k=1,2)$ by
\begin{equation*}\label{q82}
  g^{(k)}_{\alpha}(t;q):=(1+q)^{\alpha}\Gamma_{q^{2}}(\alpha+1)\,(t/2)^{-\alpha}\,J_{\alpha}^{(k)}(t(1-q);q^{2})=\mathcal{J}_{\alpha}^{(k)} (t(1-q);q^{2}),
\end{equation*}
and
\begin{equation*}
  g^{(3)}_{\alpha}(t;q):=(1+q)^{\alpha}\Gamma_{q^{2}}(\alpha+1)\,( q^{\frac{-1}{4}}t/2)^{-\alpha}\,J_{\alpha}^{(3)}(\frac{t}{2}(1-q)q^{\frac {-1}{4}};q^{2})=\mathcal{J}_{\alpha}^{(3)}(\frac {t}{2}(1-q)q^{\frac{-1}{4}};q^{2}).
\end{equation*}
}
\end{definition}
{\small
 Since the generating functions in (\ref{q8}), (\ref{q13248}), and  (\ref{q439}) can be written as}
{\small\begin{equation}\label{q96543}
\begin{split}
&\frac{te_{q}(xt)e_{q}(\frac{-t}{2})}{2\sinh_{q}\frac{t}{2}}=\sum_{n=0}^{\infty}b_{n}(x;q)\frac{t^{n}}{[n]_{q}!},\\&
\frac{t E_{q}(xt)E_{q}(\frac{-t}{2})}{2Sinh_{q}\frac{t}{2}}=\sum_{n=0}^{\infty}B_{n}(x;q)\frac{t^{n}}{[n]_{q}!},\end{split}
\end{equation}}
{\small\begin{align}\label{q103248}
 \begin{split}
  \frac{e_{q}(xt)e_{q}(\frac{-t}{2})}{\cosh_{q}\frac{t}{2}}&=\sum_{n=0}^{\infty}e_{n}(x;q)\frac{t^{n}}{[n]_{q}!},\\
  \frac{ E_{q}(xt)E_{q}(\frac{-t}{2})}{Cosh_{q}\frac{t}{2}}&=\sum_{n=0}^{\infty}E_{n}(x;q)\frac{t^{n}}{[n]_{q}!},\end{split}
\end{align}}
and
{\small\begin{equation}\label{q4039}
\begin{split}
  &\frac{t \exp_{q}(xt)\exp_{q}(\frac{-t}{2})}{2Sh_{q}(\frac{t}{2})}=\sum_{n=0}^{\infty}\tilde{B}_{n}(x;q)\frac{t^{n}}{[n]_{q}!},\\&
   \frac{ \exp_{q}(xt)\exp_{q}(\frac{-t}{2})}{Ch_{q}(\frac{t}{2})}=\sum_{n=0}^{\infty}\tilde{E}_{n}(x;q)\frac{t^{n}}{[n]_{q}!},
   \end{split}
\end{equation} }
{\small then, if we substitute with $\alpha=\pm\frac{1}{2}$  in (\ref{q14}), (\ref{q13}), and (\ref{q197}), we obtain the $q$-Bernoulli and Euler polynomials defined in (\ref{q96543}), (\ref{q103248}) and (\ref{q4039}), respectively.}

\begin{lemma}\label{gq}{\small
For $n\in\mathbb{N}_{0}$ and $Re \,\alpha>-1$,
\begin{equation*}
\frac {e_{q}(\frac{-t}{2})}{ g^{(1)}_{\alpha}(it;q)}=\frac {E_{q}(\frac{-t}{2})}{ g^{(2)}_{\alpha}(it;q)},\,\,\, |t|<\frac{1}{1-q}\min\{ j^{(1)}_{1,\alpha}, j^{(2)}_{1,\alpha},2\}.
\end{equation*}}
\end{lemma}
\begin{proof}{\small
 Hahn  in \cite{HN} proved the identity
\begin{equation}\label{q15217}
  J^{(2)}_{\alpha}(t;q)=(\frac{-t^{2}}{4};q)_{\infty}\,J^{(1)}_{\alpha}(t;q),\,\,\, |t|<2.
\end{equation}
Since
\begin{equation}\label{q8432}
\begin{split}
  &g^{(k)}_{\alpha}(it;q)=(1+q)^{\alpha}\Gamma_{q^{2}}(\alpha+1)\,(\frac {it}{2})^{-\alpha}\,J_{\alpha}^{(k)}(it(1-q);q^{2})\,\, (k=1,2),\end{split}
\end{equation}
then, substituting from (\ref{q8432}) into (\ref{q15217}), we conclude that
 \begin{equation}\label{q1698}
\begin{split}
g^{(2)}_{\alpha}(it;q)=(\frac{t^{2}}{4}(1-q)^{2};q^{2})_{\infty}\,{g^{(1)}_{\alpha}(it;q)}= E_{q}(\frac{t}{2})E_{q}(\frac{-t}{2})\, {g^{(1)}_{\alpha}(it;q)}.\end{split}
\end{equation}
Hence
\begin{equation*}
\begin{split}
\frac {E_{q}(\frac{-t}{2})}{ g^{(2)}_{\alpha}(it;q)}= \frac {E_{q}(\frac{-t}{2})}{ E_{q}(\frac{t}{2})E_{q}(\frac{-t}{2})\, g^{(1)}_{\alpha}(it;q)}=\frac {e_{q}(\frac{-t}{2})}{ g^{(1)}_{\alpha}(it;q)},\end{split}
\end{equation*}
which completes the proof.
}
\end{proof}
\begin{definition}{\small
 The generalized $q$-Bernoulli numbers ${\beta}_{n,\alpha}(q),$  $ \beta^{(3)}_{n,\alpha}(q)$
are defined respectively in terms of the generating functions
\begin{align}\label{q66}
 \frac{e_{q}(\frac{-t}{2})}{g^{(1)}_{\alpha}(it;q)}&=\frac{E_{q}(\frac{-t}{2})}{g^{(2)}_{\alpha}(it;q)} =\sum_{n=0}^{\infty}\beta_{n,\alpha}(q)\frac{t^{n}}{[n]_{q}!},\\
 \frac{\exp_{q}(\frac{-t}{2})}{g^{(3)}_{\alpha}(it;q)} &=\sum_{n=0}^{\infty}\beta^{(3)}_{n,\alpha}(q)\frac{t^{n}}{[n]_{q}!}.
\end{align}}
\end{definition}
\begin{proposition}{\small
For $n\in\mathbb{N}$, we have
\begin{equation*}
  B^{(k)}_{2n+1,\alpha}(\frac{1}{2};q)=0\quad (k=1,2,3).
\end{equation*}}
\end{proposition}
\begin{proof}{\small
 If we substitute with $x=\frac{1}{2}$  in Equations (\ref{q14})-(\ref{q197}),  we find that their left hand side are even functions. Therefore, the  coefficients of  the odd powers of $t^{n}$ on the right hand sides of Equations (\ref{q14})-(\ref{q197}) vanish. This proves the proposition.}
\end{proof}
\begin{proposition}{\small
For ${\small k\in\{1,2,3\}}$ and $n\in\mathbb{N}$,  the  polynomials $B_{n,\alpha}^{(k)}(x;q)$ have the representation  $B_{0,\alpha}^{(k)}(x;q)=1$,
\begin{equation}\label{q19}
  B^{(1)}_{n,\alpha}(x;q)=\sum_{k=0}^{n}\left[
                                    \begin{array}{c}
                                      n \\
                                      k \\
                                    \end{array}
                                  \right]_{q}
   \beta_{n-k,\alpha} (q) x^{k},
\end{equation}
\begin{equation}\label{q18}
\quad\quad\, B^{(2)}_{n,\alpha}(x;q)=\sum_{k=0}^{n}\left[
                                    \begin{array}{c}
                                      n \\
                                      k \\
                                    \end{array}
                                  \right]_{q}
    q^{\frac{k(k-1)}{2}}\beta_{n-k,\alpha}(q) x^{k},
\end{equation}
\begin{equation}\label{q180}
\quad\quad\, B^{(3)}_{n,\alpha}(x;q)=\sum_{k=0}^{n}\left[
                                    \begin{array}{c}
                                      n \\
                                      k \\
                                    \end{array}
                                  \right]_{q}
    q^{\frac{k(k-1)}{4}}\beta^{(3)}_{n-k,\alpha}(q) x^{k}.
\end{equation}}
\end{proposition}
\begin{proof}{\small
We prove the case $(k=1)$. The proofs for $(k=2,3)$ are similar and are omitted. Substituting with the series representation of $e_{q}(x)$ from (\ref{h8})
 into (\ref{q14}) gives
\begin{equation*}
 \begin{split}
 \sum_{n=0}^{\infty}B^{(1)}_{n,\alpha}(x;q)\frac{t^{n}}{[n]_{q}!}&=\frac{e_{q}(\frac{-t}{2})}{g^{(1)}_{\alpha}(it;q)}e_{q}(xt)\\
 & =\left(\sum_{n=0}^{\infty} \beta_{n,\alpha}(q)\frac{t^{n}}{[n]_{q}!}\right)\left(\sum_{n=0}^{\infty}\frac{(xt)^{n}}{[n]_{q}!}\right).
\end{split}
\end{equation*}
Hence
\begin{equation}\label{q333}
  \sum_{n=0}^{\infty}B^{(1)}_{n,\alpha}(x;q)\frac{t^{n}}{[n]_{q}!}= \sum_{n=0}^{\infty}\frac{t^{n}}{[n]_{q}!}\sum_{k=0}^{n}   \left[
                                                                \begin{array}{c}
                                                                  n \\
                                                                  k \\
                                                                \end{array}
                                                              \right]_{q}
              \beta_{n-k,\alpha}\,x^{k},\\
\end{equation}
where we applied the Cauchy product formula. Equating the $ nth $  power of $ t $ in (\ref{q333}), we obtain (\ref{q19}).}
\end{proof}
\begin{proposition}\label{ty}{\small
For ${\small n\in\mathbb{N}}$ and ${\small k\in\{1,2,3\}}$, the polynomials $B_{n,\alpha}^{(k)}(x;q)$  satisfy the $q$-difference equations
\begin{align}\label{q4504}
  D_{q,x}\,B^{(1)}_{n,\alpha}(x;q)&=[n]_{q}\,B^{(1)}_{n-1,\alpha}(x;q),\\
  D_{q^{-1},x}\,B^{(2)}_{n,\alpha}(x;q)&=[n]_{q}\,B^{(2)}_{n-1,\alpha}(x;q),\\
  \delta_{q,x}\,B^{(3)}_{n,\alpha}(x;q)&=[n]_{q}B^{(3)}_{n-1,\alpha}(x;q).
   \end{align}}
 \end{proposition}
\begin{proof}{\small
We only prove the case ($k=1$) and the proofs of $(k=2,3)$ are similar.
Calculating the $q$-derivative of both sides of (\ref{q14}) with respect to the variable $ x $ and taking into consideration that
\begin{equation*}
D_{q,x}\,e_{q}(xt)=t\,e_{q}(xt),
\end{equation*}
we obtain
\begin{equation*}
 \frac{t e_{q}(xt)e_{q}(\frac{-t}{2})}{g^{(1)}_{\alpha}(it;q)}=\sum_{n=1}^{\infty}D_{q,x} B^{(1)}_{n,\alpha}(x;q)\frac{t^{n}}{[n]_{q}!}.
 \end{equation*}
Therefore,
\begin{equation}\label{q222}
  \sum_{n=0}^{\infty} B^{(1)}_{n,\alpha}(x;q)\frac{t^{n+1}}{[n]_{q}!}=\sum_{n=1}^{\infty}D_{q,x} B^{(1)}_{n,\alpha}(x;q)\frac{t^{n}}{[n]_{q}!}.
\end{equation}
Equating the corresponding $ nth $ power of $t$ in (\ref{q222}), we obtain (\ref{q4504}). }
\end{proof}
\begin{corollary}
{ \small
Let $n\in\mathbb{N}$ and $k$ be a positive integer such that  $k\leq n$. Then for $x\in\mathbb{C},$
\begin{align*}\label{q2203}
\begin{split}
  D^{k}_{q,x}\,\frac{B^{(1)}_{n,\alpha}(x;q)}{[n]_{q}!}&=\frac{B^{(1)}_{n-k,\alpha}(x;q) }{[n-k]_{q}!},\\
  D^{k}_{{\small q^{-1},x}}\,\frac{B^{(2)}_{n,\alpha}(x;q)}{[n]_{q}!}&=\frac{B^{(2)}_{n-k,\alpha}(x;q) }{[n-k]_{q}!},\\
 \delta^{k}_{q,x}\,\frac{B^{(3)}_{n,\alpha}(x;q)}{[n]_{q}!}&=\frac{B^{(3)}_{n-k,\alpha}(x;q)}{[n-k]_{q}!}.\end{split}
 \end{align*}
  }
\end{corollary}
\begin{proof}{\small
The proofs follow from Proposition \ref{ty} and the mathematical induction.}
\end{proof}
\begin{proposition}{\small
 For $|t|<\frac{1}{1-q}\min \{ j^{(1)}_{1,\alpha}, j^{(2)}_{1,\alpha},2\}$,
\begin{equation}\label{q76}
\sum_{n=0}^{\infty} B^{(1)}_{n,\alpha}(\frac{1}{2};q)\frac{t^{n}}{[n]_{q}!}=\frac{1}{g^{(2)}_{\alpha}(it;q)}.
\end{equation}
\begin{equation}\label{q9322}
\sum_{n=0}^{\infty} B^{(2)}_{n,\alpha}(\frac{1}{2};q)\frac{t^{n}}{[n]_{q}!}=\frac{1}{g^{(1)}_{\alpha}(it;q)}.
\end{equation}}
\end{proposition}
\begin{proof}{\small
Set $  x=\frac{1}{2} $  in  (\ref{q14}), we obtain
\begin{equation}\label{q125}
 \frac{e_{q}(\frac{t}{2})e_{q}(\frac{-t}{2})}{g^{(1)}_{\alpha}(it;q)}=\sum_{n=0}^{\infty}B^{(1)}_{n,\alpha}(\frac{1}{2};q)\frac{t^{n}}{[n]_{q}!}.
\end{equation}
Substituting from (\ref{q1698}) into (\ref{q125}),  we obtain (\ref{q76}). Similarly, we can prove (\ref{q9322}).}
\end{proof}
{\small
The following Lemma from  \cite{marim} gives  the reciprocal of $ \exp_{q}(z)$  in a certain domain.}
{\small
\begin{lemma}\label{ii}
Let $z\in \Omega$,  $\displaystyle \Omega:= \{ z\in \mathbb{C}: \, |1-exp_q(-z)|< 1\}$. Then
\begin{equation*}
\dfrac{1}{exp_q(z)}:=\sum_{n=0}^{\infty} c_n\,z^{n},
\end{equation*}
where
\begin{equation}\label{q6403}
 c_n=\, \sum_{k=1}^{n} (-1)^{k}\sum_{s_1+s_2+\ldots+s_k=n\atop  s_i>0\,(i=1,\ldots,k) } \dfrac{ q^{ \sum_{i=1}^k s_i(s_i-1)/4 }}{[s_1]_q! [s_2]_q!\ldots [s_{k}]_q!}.
 \end{equation}
\end{lemma}
}
\begin{proposition}{\small
For $ Re\,\alpha >-1 $ and $t\in\displaystyle \Omega= \{ t\in \mathbb{C}: \, |1-exp_q(-t)|< 1\}$,
\begin{equation}\label{q7182}
\frac{1}{g^{(3)}_{\alpha}(it;q)}=\sum_{n=0}^{\infty}t^{n} \sum_{k=0}^{n} \frac{(-1)^{k}\,c_{k}}{2^{k}[n-k]_{q}!}\beta^{(3)}_{n-k,\alpha}(q),
\end{equation}
where $c_{n}$ is defined in (\ref{q6403}).
}
\end{proposition}
\begin{proof}{\small
Substitute with $ x=0$ in  Equation (\ref{q197}). This gives
\begin{equation*}
   \frac{\exp_{q}(\frac{-t}{2})}{g^{(3)}_{\alpha}(it;q)}=\sum_{n=0}^{\infty}\beta^{(3)}_{n,\alpha}(q)\frac{t^{n}}{[n]_{q}!}.
\end{equation*}
From Lemma \ref{ii},
\begin{equation*}
\begin{split}
   \frac{1}{g^{(3)}_{\alpha}(it;q)}&=\frac{1}{\exp_{q}(\frac{-t}{2})}\sum_{n=0}^{\infty}\beta^{(3)}_{n,\alpha}(q)\frac{t^{n}}{[n]_{q}!}\\&=
 \left( \sum_{n=0}^{\infty}c_{n}\frac{(-1)^{n}t^{n}}{2^{n}}\right)\left( \sum_{n=0}^{\infty}\beta^{(3)}_{n,\alpha}(q)\frac{t^{n}}{[n]_{q}!}\right).
   \end{split}
\end{equation*}
Applying the Cauchy product formula, we obtain (\ref{q7182}) and  completes the proof.}
\end{proof}

\begin{theorem}{\small
For $ n\in\mathbb{N}_{0} $ and $x\in\mathbb{C}$,
\begin{equation*}
  \sum  _{k =0}^{n}\left[
                                                                \begin{array}{c}
                                                                  n \\
                                                                  k \\
                                                                \end{array}
                                                              \right]_{q}
             B^{(1)}_{k,\alpha}  (-x;q) B^{(2)}_{n-k,\alpha} (x;q)=  \sum  _{k =0}^{n}\left[
                                                                \begin{array}{c}
                                                                  n \\
                                                                  k \\
                                                                \end{array}
                                                              \right]_{q}
                        \beta _{k,\alpha}(q)  \beta _{n-k,\alpha} (q).
\end{equation*}
}
\end{theorem}
\begin{proof}{\small
If we replace $ x $ by $ -x $ in (\ref{q14}), then
\begin{equation}\label{q22}
   \frac{e_{q}(-xt)e_{q}(\frac{-t}{2})}{g^{(1)}_{\alpha}(it;q)}=\sum_{n=0}^{\infty}B^{(1)}_{n,\alpha}(-x;q)\frac{t^{n}}{[n]_{q}!}.
\end{equation}
Since $ e_{q}(-xt)E_{q}(xt)=1,$ then multiplying (\ref{q13}) by (\ref{q22}) gives
\begin{equation*}
   \frac{E_{q}(\frac{-t}{2})e_{q}(\frac{-t}{2})}{g^{(2)}_{\alpha}(it;q)g^{(1)}_{\alpha}(it;q)}=\left(\sum_{n=0}^{\infty}
   B^{(1)}_{n,\alpha}(-x;q)\frac{t^{n}}{[n]_{q}!}\right)\left(\sum_{n=0}^{\infty}B^{(2)}_{n,\alpha}(x;q)\frac{t^{n}}{[n]_{q}!}\right).
\end{equation*}
From (\ref{q66}), we obtain
\begin{equation*}
 \left( \sum_{n=0}^{\infty} \beta_{n,\alpha}(q)\frac{t^{n}}{[n]_{q}!}\right)^{2}=\sum_{n=0}^{\infty}\frac{t^{n}}{[n]_{q}!}\sum_{k=0}^{n}   \left[
                                                                \begin{array}{c}
                                                                  n \\
                                                                  k \\
                                                                \end{array}
                                                              \right]_{q}
          B^{(1)}_{k,\alpha}  (-x;q) B^{(2)}_{n-k,\alpha} (x;q).
\end{equation*}
Hence
\begin{equation}\label{q24}
  \sum_{n=0}^{\infty}\frac{t^{n}}{[n]_{q}!}\sum_{k=0}^{n}   \left[
                                                                \begin{array}{c}
                                                                  n \\
                                                                  k \\
                                                                \end{array}
                                                              \right]_{q}
           \beta_{k,\alpha}(q)   \beta_{n-k,\alpha}(q) =  \sum_{n=0}^{\infty}\frac{t^{n}}{[n]_{q}!}\sum_{k=0}^{n}   \left[
                                                                \begin{array}{c}
                                                                  n \\
                                                                  k \\
                                                                \end{array}
                                                              \right]_{q}
          B^{(1)}_{k,\alpha}  (-x;q) B^{(2)}_{n-k,\alpha} (x;q).
\end{equation}
So, equating the $n$th power of $ t $ in (\ref{q24}), we obtain the required result.}
\end{proof}
\begin{proposition}\label{Ber.q and 1/q}{\small
For  $ n\in\mathbb{N}_{0},$  $ x\in\mathbb{C}$ and $q\neq 0, $
\begin{equation}\label{q217}
   B^{(2)}_{n,\alpha}  (x;q)=q^{\frac{n(n-1)}{2}} B^{(1)}_{n,\alpha} (x;\frac{1}{q}).
\end{equation}
In particular,
\begin{equation}\label{q1012}
  \beta_{n,\alpha}  (q)=q^{\frac{n(n-1)}{2}} \beta_{n,\alpha} (\frac{1}{q}).
\end{equation}}
\end{proposition}
\begin{proof}{\small
Replacing $ q$ by $\frac{1}{q}$ on the generating function in (\ref{q14}) and use  $ E_{q}(x)=e_{\frac{1}{q}}(x)$,  we obtain
\begin{equation}\label{g218}
   \frac{E_{q}(xt)E_{q}(\frac{-t}{2})}{g^{(1)}_{\alpha}(it;\frac{1}{q})}=\sum_{n=0}^{\infty}B^{(1)}_{n,\alpha}(x;\frac{1}{q})
   \frac{t^{n}}{[n]_{\frac{1}{q}}!}.
\end{equation}
Since
\begin{equation*}\label{q908}
\begin{split}
  g_{\alpha}^{(1)}(it;\frac{1}{q})&=\sum_{n=0}^{\infty}\frac{(1-q^{-1})^{2n}(\frac{t}{2})^{2n}}
 {(q^{-2};q^{-2})_{n}(q^{-2\alpha-2};q^{-2})_{n}}\\&=\sum_{n=0}^{\infty}\frac{(1-q)^{2n}q^{2n(n+\alpha )(\frac{t}{2})^{2n}}}{(q^{2},q^{2\alpha+2};q^{2})_{n}}=g_{\alpha}^{(2)}(it;q),
 \end{split}
\end{equation*}
where we used the identity $(a;q^{-1})_{n}=(a^{-1};q)_{n}(-a)^{n}q^{-\frac{n(n-1)}{2}}$. Since $ [n]_{1/q}!=q^{\frac{n(1-n)}{2}}[n]_{q}!$, then
(\ref{g218}) takes the form

\begin{equation*}
 \frac{E_{q}(xt)E_{q}(\frac{-t}{2})}{g^{(2)}_{\alpha}(it;q)}=\sum_{n=0}^{\infty}B^{(1)}_{n,\alpha}(x;\frac{1}{q})q^{\frac{n(n-1)}{2}}\frac{t^{n}}{[n]_{q}!}.
\end{equation*}
Therefore,
\begin{equation}\label{q225}
   \sum_{n=0}^{\infty}B^{(2)}_{n,\alpha}(x;q)\frac{t^{n}}{[n]_{q}!}=\sum_{n=0}^{\infty}B^{(1)}_{n,\alpha}(x;\frac{1}{q})\frac{q^{\frac{n(n-1)}{2}}t^{n}}{[n]_{q}!}.
\end{equation}
Equating  the coefficients of $ t^{n}$ in (\ref{q225}) gives (\ref{q217}) and substituting with $ x=0 $ into (\ref{q217}) yields directly (\ref{q1012}).}
\end{proof}
{\small Al-Salam, in \cite{AL-Salam},  introduced the polynomials
\begin{equation}\label{q25}
  H_{n}(x):=\sum_{k=0}^{n} \,                          \left[
                                                                \begin{array}{c}
                                                                  n \\
                                                                  k \\
                                                                \end{array}
                                                              \right]_{q}\,x^{k} ,\quad G_{n}(x):=\sum_{k=0}^{n} \, \left[
                                                                \begin{array}{c}
                                                                  n \\
                                                                  k \\
                                                                \end{array}
                                                              \right]_{q}\, q^{k^{2}-nk}  x^{k}.
\end{equation}
 He also proved that
\begin{equation}\label{q26}
  E_{q}(x) E_{q}(-x)=\sum_{n=0}^{\infty} q^{\frac{n(n-1)}{2}} G_{n}(-1)\frac{x^{n}}{[n]_{q}!} ,\quad x \in \,\mathbb{C},
\end{equation}
\begin{equation}\label{q27}
   e_{q}(x) e_{q}(-x)=\sum_{n=0}^{\infty}  H_{n}(-1)\frac{x^{n}}{[n]_{q}!} ,\quad |x|\,<\,\frac{1}{1-q}.
\end{equation}}
\quad
\quad
{\small The following theorem introduces connection relations between  the polynomials $B^{(1)}_{n,\alpha}(x;q)$ and $B^{(2)}_{n,\alpha}(x;q)$.}
\begin{theorem}{\small
For $ n \in \mathbb{N}_{0} $,
\begin{align}\label{q29}
   B^{(1)}_{n,\alpha}(x;q)&=\sum_{k=0}^{n} \, \left[
                                                                \begin{array}{c}
                                                                  n \\
                                                                  k \\
                                                                \end{array}
                                                              \right]_{q}\, x^{k}H_{k}(-1)  B^{(2)}_{n-k,\alpha}(x;q),\\
  B^{(2)}_{n,\alpha}(x;q)&=\sum_{k=0}^{n} \, \left[
                                                                \begin{array}{c}
                                                                  n \\
                                                                  k \\
                                                                \end{array}
                                                              \right]_{q}\, q^{\frac{k(k-1)}{2}}x^{k}G_{k}(-1)   B^{(1)}_{n-k,\alpha}(x;q).
\end{align}}
\end{theorem}
\begin{proof}{\small
Since ${\small E_{q}(xt) e_{q}(-xt)= 1,\,|xt|<\frac{1}{1-q},}$ then  from (\ref{q1698}), the generating function of $ B^{(1)}_{n,\alpha}(x;q) $ can be represented as
\begin{equation*}
\begin{split}
  \frac{e_{q}(xt)e_{q}(\frac{-t}{2})}{g^{(1)}_{\alpha}(it;q)}=\frac{E_{q}(xt) E_{q}(\frac{-t}{2})}{g^{(2)}_{\alpha}(it;q)}e_{q}(xt)e_{q}(-xt).
\end{split}
\end{equation*}
From (\ref{q14}), (\ref{q13}) and (\ref{q27}), we obtain
\begin{equation} \label{q30}
\begin{split}
  \sum_{n=0}^{\infty}B^{(1)}_{n,\alpha}(x;q)\frac{t^{n}}{[n]_{q}!}&= \left(\sum_{n=0}^{\infty}B^{(2)}_{n,\alpha}(x;q)\frac{t^{n}}{[n]_{q}!}\right)\left(\sum_{n=0}^{\infty}  H_{n}(-1)\frac{(xt)^{n}}{[n]_{q}!}\right)
  \\& =\sum_{n=0}^{\infty} \frac{t^{n}}{[n]_{q}!} \sum_{k=0}^{n} \, \left[
                                                                \begin{array}{c}
                                                                  n \\
                                                                  k \\
                                                                \end{array}
                                                              \right]_{q}\, x^{k}H_{k}(-1)  B^{(2)}_{n-k,\alpha}(x;q).\end{split}
\end{equation}
Therefore, equating the coefficients of the $ nth $ power of $ t $   in the series of the outside
parts of  (\ref{q30}) gives (\ref{q29}). The proof for $B^{(2)}_{n,\alpha}(x;q)$ follows similarly from
the generating function of  $ B^{(2)}_{n,\alpha}(x;q) $  and the identity (\ref{q26}), and is omitted.}
\end{proof}
\begin{theorem}{\small
Let $n$ be a positive integer and $x,\,\alpha $ be a complex numbers such that\\ $Re\,\alpha>-1 $.  Then
\begin{align}\label{q102}
\begin{split}
 \sum_{k = 0}^{[\frac{n}{2}]} \, \frac{ (1-q)^{2k} \,B^{(1)}_{n-2k,\alpha}(-\frac{x}{2};q)}{2^{2k}\,[n-2k]_{q}!\, (q^{2},q^{2n+\alpha};q^{2})_{k}}&=\frac{(-1/2)^{n}}{[n]_{q}!} H_{n}(x),\\
 \sum_{k =0 }^{[\frac{n}{2}]} \, \frac{ (1-q)^{2k} q^{2k(k+\alpha)}\,B^{(2)}_{n-2k,\alpha}(-\frac{x}{2};q)}{2^{2k}\,[n-2k]_{q}!\, (q^{2},q^{2n+\alpha};q^{2})_{k}}&=\frac{(-1/2)^{n}}{[n]_{q}!}q^{\frac{n(n-1)}{2}} G_{n}(x),\\
  \sum_{k = 0}^{[\frac{n}{2}]} \, \frac{ (1-q)^{2k} q^{k^{2}+k/2}\,B^{(3)}_{n-2k,\alpha}(-\frac{x}{2};q)}{2^{2k}\,[n-2k]_{q}!\, (q^{2},q^{2\alpha+2};q^{2})_{k}}&=\frac{(-1/2)^{n}}{[n]_{q}!}q^{\frac{n(n-1)}{4}} (-xq^{\frac{1-n}{2}};q)_{n} .
  \end{split}
\end{align}
}
\end{theorem}
\begin{proof}{\small
We can write the generating function of the polynomials $B^{(1)}_{n,\alpha}(x;q)$ as
\begin{equation}\label{q35}\begin{split}
  e_{q}(xt)e_{q}(\frac{-t}{2}) &= g^{(1)}_{\alpha}(it;q)\sum_{n=0}^{\infty} B^{(1)}_{n,\alpha}(x;q)\frac{t^{n}}{[n]_{q}!}\\&=\left(\sum_{n=0}^{\infty} \frac{(1-q)^{2n}t^{2n}}{2^{2n}(q^{2},q^{2\alpha+2};q^{2})_{n}}\right)\,\left(\sum_{n=0}^{\infty} B^{(1)}_{n,\alpha}(x;q)\frac{t^{n}}{[n]_{q}!}\right).\end{split}
\end{equation}

On one hand, applying the Cauchy product formula in (\ref{q35}), we obtain
\begin{equation*}
  e_{q}(xt)e_{q}(\frac{-t}{2})=\sum_{n=0}^{\infty}t^{n}\sum_{k=0}^{[\frac{n}{2}]}\frac{(1-q)^{2k}B^{(1)}_{n-2k,\alpha}(x;q)}{2^{2k}[n-2k]_{q}!(q^{2},q^{2\alpha+2};q^{2})_{k}}
  .
\end{equation*}
On the other hand, using the series representation of $e_{q}(x)$  in (\ref{h8}) followed by the Cauchy product formula, and using (\ref{q25}) yields
\begin{equation}\label{q307}
  e_{q}(xt)e_{q}(\frac{-t}{2})
  =\sum_{n=0}^{\infty}\frac{t^{n}}{[n]_{q}!}\,(\frac{-1}{2})^{n} H_{n}(-2x).
\end{equation}
Hence
\begin{equation}\label{r4}
  \sum_{n=0}^{\infty}\frac{t^{n}}{[n]_{q}!}\,(\frac{-1}{2})^{n} H_{n}(-2x)=\sum_{n=0}^{\infty}t^{n}\sum_{k=0}^{[\frac{n}{2}]}\frac{(1-q)^{2k}B^{(1)}_{n-2k,\alpha}(x;q)}{2^{2k}[n-2k]_{q}!(q^{2},q^{2\alpha+2};q^{2})_{k}}
  ,
\end{equation}
equating the coefficients of $t^{n}$ in (\ref{r4}), we get
\begin{equation}\label{q6438}
  \frac{(\frac{-1}{2})^{n}\,H_{n}(-2x)}{[n]_{q}!}=\sum_{k=0}^{[\frac{n}{2}]}\frac{(1-q)^{2k}B^{(1)}_{n-2k,\alpha}(x;q)}{2^{2k}[n-2k]_{q}!(q^{2},q^{2\alpha+2};q^{2})_{k}}.
\end{equation}
Replacing $x \,by\,\, \frac{-x}{2}$  in (\ref{q6438}) gives
\begin{equation*}
  \frac{(\frac{-1}{2})^{n}\,H_{n}(x)}{[n]_{q}!}=\sum_{k=0}^{[\frac{n}{2}]}\frac{(1-q)^{2k}B^{(1)}_{n-2k,\alpha}(\frac{-x}{2};q)}{2^{2k}[n-2k]_{q}!(q^{2},q^{2\alpha+2};q^{2})_{k}},
\end{equation*}
which readily completes the proof for $B^{(1)}_{n,\alpha}(x;q)$. The proofs for $B^{(2)}_{n,\alpha}(x;q)$  and  $B^{(3)}_{n,\alpha}(x;q)$ are similar and are omitted.}
\end{proof}

{\small If we set $x=0$ in (\ref{q102}), we obtain the following recurrence relations for $\beta_{n,\alpha}(q)$ and $\beta^{(3)}_{n,\alpha}(q),$}
{\small\begin{align}\label{q1902}
\begin{split}
 \sum_{k = 0}^{[\frac{n}{2}]} \, \frac{ (1-q)^{2k} \,\beta_{n-2k,\alpha}(q)}{2^{2k}\,[n-2k]_{q}!\, (q^{2},q^{2\alpha+2};q^{2})_{k}}\,&=\frac{(\frac{-1}{2})^{n}}{[n]_{q}!},\quad n\in\mathbb{N},\\
 \quad\sum_{k = 0}^{[\frac{n}{2}]} \, \frac{(1-q)^{2k} q^{k^{2}+k/2}\,\beta^{(3)}_{n-2k,\alpha}(q)}{2^{2k}\,[n-2k]_{q}!\, (q^{2},q^{2\alpha+2};q^{2})_{k}}\,&=\frac{q^{\frac{n(n-1)}{4}}(\frac {-1}{2})^{n}}{[n]_{q}!},\quad n\in\mathbb{N}.
 \end{split}
\end{align}}
{\small As a consequence of the recursive relations in (\ref{q1902}), and the fact that
\begin{equation*}
  \beta_{0,\alpha}(q)=\beta^{(3)}_{0,\alpha}(q)=1,
\end{equation*}
we can prove that}
\\

{\small\begin{align*}
\beta_{1,\alpha}(q)&=-\frac{1}{2},\quad \beta_{2,\alpha}(q)=\frac{q(1-q^{2\alpha+1})}{4(1-q^{2\alpha+2})},\quad
\beta_{3,\alpha}(q)=\frac{-q^{3}(1-q^{2\alpha-1})}{8(1-q^{2\alpha+2})},\\
\\
\beta_{4,\alpha}(q)&\nonumber=\frac{1}{16}-\frac{(q+q^{3})(1-q^{3})(1-q^{2\alpha+1})}{16(1-q^{2\alpha+2})^{2}}-\frac{(1-q)(1-q^{3})}
{16(q^{2\alpha+2};q^{2})_{2}},\\
\\
\beta_{5,\alpha}(q)&\nonumber=\frac{(1+q^{2})(1-q^{5})(q^{3}-q^{2\alpha+2})}{32(1-q^{2\alpha+2})^{2}}+\frac{(1-q^{3})(1-q^{5})}{32(q^{2\alpha+2};q^{2})_{2}}
-\frac{1}{32},
\end{align*}}
and
{\small\begin{equation*}
\begin{split}
&\beta^{(3)}_{1,\alpha}(q)=\frac{-1}{2},\quad\quad \beta^{(3)}_{2,\alpha}(q)= \frac{q^{1/2}(1-q^{2\alpha+2})-q^{3/2}(1-q)}{4(1-q^{2\alpha+2})},\\&
\beta^{(3)}_{3,\alpha}(q)=\frac{-q^{3/2}(q^{3}-q^{2\alpha+2})}{8(1-q^{2\alpha+2})},
\end{split}
\end{equation*}
\begin{equation*}
\begin{split}
\beta^{(3)}_{4,\alpha}(q)=&\frac{q^{3}(q^{2\alpha+2};q^{2})_{2}(1-q^{2\alpha+2})
}{16(1-q^{2\alpha+2})^{2}(1-q^{2\alpha+4})}-\frac{[3]_{q}q^{5}(1-q)^{2}(1-q^{2\alpha+2})}{16(1-q^{2\alpha+2})^{2}(1-q^{2\alpha+4})}\\&+\frac{[4]_{q}[3]_{q}q^{3/2}
(1-q^{2\alpha+4})\left(q^{1/2}(1-q^{2\alpha+2})-q^{3/2}(1-q)\right)}{16(1-q^{2\alpha+2})^{2}(1-q^{2\alpha+4})},
\end{split}
\end{equation*}
\quad
\begin{equation*}
\begin{split}
\quad\beta^{(3)}_{5,\alpha}(q)=&\frac{[5]_{q}q^{3}(1-q)(1+q^{2})(q^{3}-q^{2\alpha+2})(1-q^{2\alpha+4})}{32(1-q^{2\alpha+2})^{2}(1-q^{2\alpha+4})} \\&+\frac {[5]_{q}q^{5}(1-q)(1-q^{3})(1-q^{2\alpha+2})}{32(1-q^{2\alpha+2})^{2}(1-q^{2\alpha+4})}-\frac {q^{5}(1-q^{2\alpha+2})^{2}(1-q^{2\alpha+4})}{32(1-q^{2\alpha+2})^{2}(1-q^{2\alpha+4})}.\end{split}
\end{equation*}}
\begin{theorem}{\small
For $  n\in\,\mathbb{N}_{0} $ and complex numbers $ a $ and $ x$,
\begin{align}\label{q691}
  B^{(1)}_{n,\alpha}(x;q)&=\sum_{k=0}^{n}\left[
                                                                \begin{array}{c}
                                                                  n \\
                                                                  k \\
                                                                \end{array}
                                                              \right]_{q}\,(a;q)_{k} \,x^{k} B^{(1)}_{n-k,\alpha}(ax;q),\\
  B^{(2)}_{n,\alpha}(x;q)&=\sum_{k=0}^{n}\left[
                                                                \begin{array}{c}
                                                                  n \\
                                                                  k \\
                                                                \end{array}
                                                              \right]_{q}\,(-a)^{k}(1/a;q)_{k} \,x^{k} B^{(2)}_{n-k,\alpha}(ax;q).
\end{align}}
\end{theorem}
\begin{proof}{\small
The proof of (\ref{q691}) follows from the generating function (\ref{q14}) since
\begin{equation*}
\begin{split}
  \frac{e_{q}(xt)e_{q}(\frac{-t}{2})}{g^{(1)}_{\alpha}(it;q)}=\frac{e_{q}(tax)e_{q}(\frac{-t}{2})}{g^{(1)}_{\alpha}(it;q)}\,\frac{e_{q}(tx)}{e_{q}(atx)},\quad |tx|<\frac{1}{1-q}.\\
\end{split}
\end{equation*}
From the $q$-binomial theorem (see \cite[ Eq.(1.3.2), P. 8]{Gasper}), we can prove that
\begin{equation*}
 \frac{e_{q}(tx)}{e_{q}(atx)}=\sum_{n=0}^{\infty}\frac{(a;q)_{n}}{(q;q)_{n}}((1-q)tx)^{n},\quad |tx|<\frac{1}{1-q}.
\end{equation*}
Therefore,
\begin{equation}\label{q78}
\begin{split}
  \sum_{n=0}^{\infty}B^{(1)}_{n,\alpha}(x;q)\frac{t^{n}}{[n]_{q}!}&=\left(\sum_{n=0}^{\infty}B^{(1)}_{n,\alpha}(ax;q)\frac{t^{n}}{[n]_{q}!}\right)\,\,
  \left(\sum_{n=0}^{\infty}\frac{(a;q)_{n}}{[n]_{q}!}(tx)^{n}\right)\\
                                                       &=\sum_{n=0}^{\infty}\frac{t^{n}}{[n]_{q}!}\sum_{k=0}^{n}\left[ \begin{array}{c}
                                                                  n \\
                                                                  k \\
                                                                \end{array}
                                                             \right]_{q}\, (a;q)_{k}x^{k} B^{(1)}_{n-k,\alpha}(ax;q),\\
\end{split}
\end{equation}
 where we used the Cauchy product formula. Equating the coefficients of $ t^{n} $ in (\ref{q78}),  we obtain (\ref{q691}). The proof for $B^{(2)}_{n,\alpha}(x;q)$ is similar and is omitted.}
\end{proof}

\begin{lemma}\label{iq}{\small
For $n\in\mathbb{N}_{0}$,  $Re \,\alpha>-1$, and $|\frac{(1-q)t}{2}|<1$,
\begin{equation}\label{h7}
 g^{(1)}_{\alpha}(it;q)E_{q}(\frac{t}{2})= g^{(2)}_{\alpha}(it;q)e_{q}(\frac{t}{2})
={}_{2}\phi_{1}\,(q^{\alpha+\frac{1}{2}},-q^{\alpha+\frac{1}{2}};q^{2\alpha+1};q,\frac{(1-q)t}{2}).
\end{equation}}
\end{lemma}
\begin{proof}{\small
From Lemma \ref{gq}, we conclude that
\begin{equation*}
\begin{split}
g^{(2)}_{\alpha}(it;q) e_{q}(\frac{t}{2})=g^{(1)}_{\alpha}(it;q)E_{q}(\frac{t}{2}).\end{split}
\end{equation*}
From the series representations of $E_{q}(x)$ and $g^{(1)}_{\alpha}(it;q)$ in (\ref{h8}) and (\ref{q35}), respectively we obtain
\begin{equation*}\label{g4}
\begin{split}
 g^{(1)}_{\alpha}(it;q)E_{q}(\frac{t}{2})
 & =\left(\sum_{n=0}^{\infty}\frac{(1-q)^{2n}t^{2n}}{2^{2n}(q^{2},q^{2\alpha+2};q^{2})_{n}}\right)\left(
\sum_{n=0}^{\infty}\frac{q^{\frac{n(n-1)}{2}}(1-q)^{n}t^{n}}{2^{n}(q;q)_{n}}\right)\\&=\sum_{n=0}^{\infty}\frac{(1-q)^{n}q^{\frac{n(n-1)}{2}}t^{n}}{2^{n}}
\sum_{k=0}^{[\frac{n}{2}]}\frac{q^{2k^{2}-2nk +k}
}{(q;q)_{n-2k}(q^{2},q^{2\alpha+2};q^{2})_{k}}\\&=\sum_{n=0}^{\infty}\frac{q^{\frac{n(n-1)}{2}}(1-q)^{n}t^{n}}{2^{n}(q;q)_{n}}\sum_{k=0}^{[\frac{n}{2}]}
\frac{q^{2k}(q^{-n};q)_{2k}
}{(q^{2},q^{2\alpha+2};q^{2})_{k}},\end{split}
\end{equation*}
 where we used the identity, see \cite[Eq. (1.2.32), P. 6]{Gasper},
{\small\begin{equation}\label{f798}
\begin{split}
  (a;q)_{n-k}=\frac{(a;q)_{n}}{(a^{-1}q^{1-n};q)_{k}}(-qa^{-1})^{k}q^{\frac{k(k-1)}{2}
  -nk}\,\,\,(k=0,1,\ldots,n).
 \end{split}
\end{equation}}
Therefore, using the identity $(a;q)_{2n}=(a;q^{2})_{n}(aq;q^{2})_{n}$ yields
\begin{equation*}
\begin{split}
  g^{(1)}_{\alpha}(it;q)E_{q}(\frac{t}{2}) &=\sum_{n=0}^{\infty}\frac{q^{\frac{n(n-1)}{2}}(1-q)^{n}(\frac{t}{2})^{n}}{(q;q)_{n}}\sum_{k=0}^{[\frac{n}{2}]} \frac{q^{2k}(q^{-n};q^{2})_{k}(q^{-n+1};q^{2})_{k}}{(q^{2};q^{2\alpha+2};q^{2})_{k}}\\&=
\sum_{n=0}^{\infty}\dfrac{q^{\frac{n(n-1)}{2}}(1-q)^{n}(\frac{t}{2})^{n}}{(q;q)_{n}}\,\,_{2}\phi_{1}(q^{-n},q^{-n+1};q^{2\alpha+2};q^{2},q^{2}).
\end{split}
\end{equation*}
Since
\begin{equation*}
 _{2}\phi_{1}\,(q^{-n},q^{1-n};qb^{2};q^{2},q^{2})=\frac{(b^{2};q^{2})_{n}}{(b^{2};q)_{n}}q^{\frac{-n(n-1)}{2}}\,\,(n\in\mathbb{N}),
\end{equation*}
see \cite[P. 26 ]{Gasper}, then
\begin{equation}\label{t2303}
\begin{split}
 g^{(1)}_{\alpha}(it;q)E_{q}(\frac{t}{2})&=\sum_{n=0}^{\infty}\dfrac{(q^{\alpha+\frac{1}{2}};q)_{n}(-q^{\alpha+\frac{1}{2}};q)_{n}(\frac{(1-q)t}{2})^{n}}{
(q;q)_{n}(q^{2\alpha+1};q)_{n}}\\&={}_{2}\phi_{1}\,(q^{\alpha+\frac{1}{2}},-q^{\alpha+\frac{1}{2}};q^{2\alpha+1};q,\frac{(1-q)t}{2}).
  \end{split}
\end{equation}
Hence from Lemma \ref{gq} and (\ref{t2303}), we obtain (\ref{h7}) and completes the proof.

}
\end{proof}

\begin{theorem}{\small
Let $ \alpha $ be a complex number such that $ Re\,\alpha >-1$. Then
\begin{align}\label{q48}
 \sum_{m=0}^{n}\left[
                                                                \begin{array}{c}
                                                                  n \\
                                                                  m \\
                                                                \end{array}
                                                              \right]_{q}\dfrac{(q^{2\alpha+1};q^{2})_{m}B^{(1)}_{n-m,\alpha}(x;q)}{2^{m}
(q^{2\alpha+1};q)_{m}}\nonumber&= x^{n},\\
 \sum_{m=0}^{n}\left[
                                                                \begin{array}{c}
                                                                  n \\
                                                                  m \\
                                                                \end{array}
                                                              \right]_{q}\dfrac{(q^{2\alpha+1};q^{2})_{m}B^{(2)}_{n-m,\alpha}(x;q)}{2^{m}
(q^{2\alpha+1};q)_{m}}\nonumber&=q^{\frac{n(n-1)}{2}}x^{n},\\
  \sum_{m=0}^{n}(-\frac{1}{2})^{m}\left(\sum_{k = 0}^{[\frac{m}{2}]}\,  \frac{
  \,q^{k^{2}+k/2}\,(1-q)^{2k}\,c_{m-2k}}{\,(q^{2},q^{2\alpha+2};q^{2})_{k}}\right)\frac{B^{(3)}_{n-m,\alpha}(x;q)}{[n-m]_{q}!}\nonumber&= \frac{q^{\frac{n(n-1)}{4}}x^{n}}{[n]_{q}!},
\end{align}}
where $(c_{k})_{k}$ are the coefficients defined in (\ref{q6403}).
\end{theorem}
\begin{proof}{\small
We can write  Equation (\ref{q14}) in the form
\begin{equation}\label{q41}
\begin{split}
  e_{q}(xt)&=E_{q}(\frac{t}{2})g^{(1)}_{\alpha}(it;q)\sum_{n=0}^{\infty} B^{(1)}_{n,\alpha}(x;q)\frac{t^{n}}{[n]_{q}!}\\& = \left(\sum_{n=0}^{\infty} d_{n} t^{n}\right)\,\left( \sum_{n=0}^{\infty} B^{(1)}_{n,\alpha}(x;q)\frac{t^{n}}{[n]_{q}!}\right).\end{split}
\end{equation}
From Lemma \ref{iq}, we obtain
\begin{equation}\label{t303}
\begin{split}
 g^{(1)}_{\alpha}(it;q)E_{q}(\frac{t}{2})=\sum_{n=0}^{\infty} d_{n}t^{n},
  \end{split}
\end{equation}
where
\begin{equation}\label{q47}
 d_{n}=\dfrac{(1-q)^{n}(q^{2\alpha+1};q^{2})_{n}}{2^{n}(q;q)_{n}
(q^{2\alpha+1};q)_{n}} .
\end{equation}

Now, applying the Cauchy product formula in (\ref{q41}) gives
\begin{equation}\label{h44}
\begin{split}
   e_{q}(xt)=\sum_{n=0}^{\infty}t^{n}\,\sum_{m=0}^{n}\frac{d_{m} B^{(1)}_{n-m,\alpha}(x;q) }{ [n-m]_{q}!}=
    \sum_{n=0}^{\infty}\frac{(xt)^{n}}{[n]_{q}!}.
\end{split}
\end{equation}
Equating the coefficients of the $n$th power of $t$ in (\ref{h44}) gives
\begin{equation}\label{q4500}
  \sum_{m=0}^{n}\frac{d_{m} B^{(1)}_{n-m,\alpha}(x;q) }{ [n-m]_{q}!}=\frac{x^{n}}{[n]_{q}!}.
\end{equation}
Substituting from (\ref{q47}) into (\ref{q4500}), we get the result for $B^{(1)}_{n,\alpha}(x;q)$. Similarly, we can prove the result for
$ B^{(k)}_{n,\alpha}(x;q)\,(k=2,3) $. }
\end{proof}
\begin{theorem}{\small
Let $ n$ be a positive integer and $x$ be a complex number. If $Re\,\alpha>-1 $, then
\begin{equation}\label{q3058}
\begin{split}
   B^{(1)}_{n,\alpha}(x;q)&-(-1)^{n} B^{(1)}_{n,\alpha}(-x;q)\\&=\sum_{k=0}^{n}\left[
   \begin{array}{c}
                                                                  n \\
                                                                  k \\
                                                                \end{array}
                                                              \right]_{q}\left((\frac{-1}{2})^{k}H_{k}(-2x)-(\frac{1}{2})^{k}H_{k}(2x)\right) B^{(2)}_{n-k,\alpha}(\frac{1}{2};q),
                                                              \end{split}
\end{equation}
\begin{equation}\label{q398}
\begin{split}
   B^{(2)}_{n,\alpha}(x;q)&-(-1)^{n} B^{(2)}_{n,\alpha}(-x;q)\\&=\sum_{k=0}^{n}\left[
   \begin{array}{c}
                                                                  n \\
                                                                  k \\
                                                                \end{array}
                                                              \right]_{q}q^{\frac{k(k-1)}{2}}\left((\frac{-1}{2})^{k}G_{k}(-2x)-(\frac{1}{2})^{k}G_{k}(2x)\right) B^{(1)}_{n-k,\alpha}(\frac{1}{2};q).
                                                              \end{split}
\end{equation}}
\end{theorem}
\begin{proof}{\small
We give only the proof of (\ref{q3058}) since the proof of (\ref{q398}) is similar. From (\ref{q14}),
\begin{equation}\label{q400}
  \frac{e_{q}(xt)e_{q}(\frac{-t}{2})}{g^{(1)}_{\alpha}(it;q)}-\frac{e_{q}(xt)e_{q}(\frac{t}{2})}{g^{(1)}_{\alpha}(-it;q)}
  =\sum_{n=0}^{\infty}B^{(1)}_{n,\alpha}(x;q)\frac{t^{n}}{[n]_{q}!}-\sum_{n=0}^{\infty}B^{(1)}_{n,\alpha}(-x;q)\frac{(-t)^{n}}{[n]_{q}!}.
\end{equation}
Since
\begin{equation*}
g^{(1)}_{\alpha}(-it;q)=g^{(1)}_{\alpha}(it;q),
\end{equation*}
then Equation (\ref{q400}) can be written as
\begin{equation}\label{q9501}
 \frac{e_{q}(xt)e_{q}(\frac{-t}{2})
 -e_{q}(xt)e_{q}(\frac{t}{2})}{g^{(1)}_{\alpha}(it;q)}
  =\sum_{n=0}^{\infty}\left [B^{(1)}_{n,\alpha}(x;q)-(-1)^{n}B^{(1)}_{n,\alpha}(-x;q)\right]\frac{t^{n}}{[n]_{q}!}.
\end{equation}
Replacing $ x, t$  by  $-x, -t$, respectively in (\ref{q307}) gives
\begin{equation}\label{q1052}
  e_{q}(xt)e_{q}(\frac{t}{2})= \sum_{n=0}^{\infty}\frac{t^{n}}{[n]_{q}!}(\frac{1}{2})^{n}H_{n}(2x).
\end{equation}
From (\ref{q307}) and (\ref{q1052}), the left hand side of  (\ref{q9501}) can be written as
\begin{equation*}\label{q403}
 \frac{e_{q}(xt)e_{q}(\frac{-t}{2}) -e_{q}(xt)e_{q}(\frac{t}{2})}{g^{(1)}_{\alpha}(it;q)}=\frac{1} {g^{(1)}_{\alpha}(it;q)} \sum_{n=0}^{\infty}\frac{t^{n}}{[n]_{q}!}\left((\frac{-1}{2})^{n}H_{n}(-2x)-(\frac{1}{2})^{n}H_{n}(2x)\right).
\end{equation*}
Therefore, by (\ref{q9322}) and the Cauchy product formula, we get
{\small\begin{equation}\label{q408}
\begin{split}
 &\frac{e_{q}(xt)e_{q}(\frac{-t}{2}) -e_{q}(xt)e_{q}(\frac{t}{2})}{g^{(1)}_{\alpha}(it;q)}=\\&\sum_{n=0}^{\infty}\frac{t^{n}}{[n]_{q}!}\sum_{k=0}^{n}
 \left[\begin{array}{c}
                                                                  n \\
                                                                  k \\
                                                                \end{array}
                                                              \right]_{q}\left((\frac{-1}{2})^{k}H_{k}(-2x)-(\frac{1}{2})^{k}H_{k}(2x)\right)
                                                              B^{(2)}_{n-k,\alpha}(\frac{1}{2};q).
 \end{split}
\end{equation}}
Since the left hand side of (\ref{q9501}) and (\ref{q408}) are equal, then equating the coefficients of $t^{n}$ on the right hand sides of (\ref{q9501}) and (\ref{q408}) yields (\ref{q3058}) and completes the proof.}
\end{proof}
\begin{proposition}\label{kk}
 If ${\small\alpha_{0}>-1}$ satisfies the condition
\begin{equation}\label{q8754}
  q^{2(\alpha_{0}+1)}(1-q)^{2}<(1-q^{2})(1-q^{2\alpha_{0}+2}),
\end{equation}
then  $(t/2)^{-\alpha}J_{\alpha}^{(2)}(t(1-q);q^{2})$ has no zeros in $|t|\leq1$ for all $\alpha\geq\alpha_{0}.$

\end{proposition}
\begin{proof}
{\small
Set
\begin{equation*}
 F(t):= \frac{(q;q)_{\infty}}{(q^{\alpha+1};q)_{\infty}}(t/2)^{-\alpha}J_{\alpha}^{(2)}(t(1-q);q^{2})
=\sum_{k=0}^{\infty}\frac{(-1)^{k}q^{2k(k+\alpha)}(1-q)^{2k}}{2^{2k}(q^{2},q^{2\alpha+2};q^{2})_{k}}t^{2k},
\end{equation*}
and
\begin{equation*}
  a_{k}:=\frac{q^{2k(k+\alpha)}(1-q)^{2k}}{2^{2k}(q^{2},q^{2\alpha+2};q^{2})_{k}}.
\end{equation*}
Then, under hypothesis (\ref{q8754}) and since $0 < q < 1$,
{\small\begin{equation*}
  q^{2(\alpha+1)}(1-q^{2})\leq q^{2(\alpha_{0}+1)}(1-q^{2})<(1-q^{2})(1-q^{2\alpha_{0}+2})\leq(1-q^{2})(1-q^{2\alpha+2}),
\end{equation*}}
holds whenever $\alpha\geq \alpha_{0}$. Hence
\begin{equation*}
   \frac{a_{k+1}}{a_{k}}=\frac{q^{4k+2(\alpha+1)}(1-q)^{2}}{4(1-q^{2k+2})(1-q^{2k+2\alpha+2})}
\leq\frac{q^{2(\alpha+1)}(1-q)^{2}}{4(1-q^{2})(1-q^{2\alpha+2})}<1,
\end{equation*}
for $t\in\mathbb{R},\, |t|\leq1$
\begin{equation*}
\begin{split}
  F(t)=\sum_{k=0}^{\infty}t^{2k}(a_{2k}-a_{2k+1}t^{2})\geq (a_{0}-a_{1}t^{2})\geq(a_{0}-a_{1})>0.
\end{split}
\end{equation*}
This proves that $F(t)$ has no zeros on $[-1,1]$, since $F(t)$ has only real zeros, then $F(t)$ has no zeros in the unit disk. i.e $|F(t)|>0,\,\,for\,\,|t|\leq1.$
}
\end{proof}
{\small\begin{corollary}\label{nn}
There exists $\alpha_{0}>-1$ such that $J_{\alpha}^{(2)}(t(1-q);q^{2})$ has no zeros in the unit disk for all $\alpha\geq\alpha_{0}$.
\end{corollary}}
\begin{proof}{\small
Since for a fixed $q\in (0,1)$,
\begin{equation*}
  \displaystyle\lim_{\alpha\rightarrow\infty}q^{2\alpha+2}=0,\,\,\displaystyle\lim_{\alpha\rightarrow\infty}(1-q^{2})(1-q^{2\alpha+2})=(1-q^{2}),
\end{equation*}
then there exists $\alpha_{0}>-1 $ such that the condition (\ref{q8754}) holds for all $\alpha\geq\alpha_{0}$. Consequently from Proposition \ref{kk},  $J_{\alpha}^{(2)}(t(1-q);q^{2})$ has no zeros in the unit disk for all $\alpha\geq\alpha_{0}$.
}
\end{proof}

\begin{theorem}{\small
For $n\in\mathbb{N}$,
\begin{equation}\label{q10046}
 \quad\quad\quad\quad\lim_{\alpha\rightarrow\infty}\,B^{(2)}_{n,\alpha}(x;q)=(-\frac{1}{2})^{n}q^{\frac{n(n-1)}{2}}G_{n}(-2x),
\end{equation}
\begin{equation}\label{q10045}
\lim_{\alpha\rightarrow\infty}\,B^{(1)}_{n,\alpha}(x;q)=x^{n}(\frac{1}{2x};q)_{n}.
\end{equation}}
\end{theorem}
\begin{proof}{\small
Taking the  limit on both sides of Equation (\ref{q13}) as $ \alpha\rightarrow\infty$ we get
\begin{equation}\label{q3563}
 \lim_{\alpha\rightarrow\infty} \frac{E_{q}(xt)E_{q}(\frac{-t}{2})}{ g^{(2)}_{\alpha}(it;q)}= \lim_{\alpha\rightarrow\infty}\,\sum_{n=0}^{\infty} B^{(2)}_{n,\alpha}(x;q)\frac{t^{n}}{[n]_{q}!}.
\end{equation}
From Corollary  \ref{nn}, there exists $\alpha_{0}>-1$ such that
 $g^{(2)}_{\alpha}(it;q)$ has no zeros in $|t|\leq1$ for all $\alpha\geq\alpha_{0}$.  This means that $\dfrac{E_{q}(xt)E_{q}(\frac{-t}{2})}{ g^{(2)}_{\alpha}(it;q)}$ is analytic in $|t|\leq1$ for all $\alpha\geq\alpha_{0}$. Therefore, we can interchange the limit with the summation in (\ref{q3563}) when $|t|\leq1$ to obtain
\begin{equation*}
  \frac{E_{q}(xt)E_{q}(\frac{-t}{2})}{\displaystyle \lim_{\alpha\rightarrow\infty} g^{(2)}_{\alpha}(it;q)}= \sum_{n=0}^{\infty} \lim_{\alpha\rightarrow\infty}\, B^{(2)}_{n,\alpha}(x;q)\frac{t^{n}}{[n]_{q}!}.
\end{equation*}
Since
\begin{equation*}
 \lim_{\alpha\rightarrow\infty} g^{(2)}_{\alpha}(it;q)=1,\quad E_{q}(xt)E_{q}(yt)=\sum_{n=0}^{\infty}\frac{q^{n(n-1)/2}(ty)^{n}}{[n]_{q}!}G_{n}(\frac{x}{y}),
\end{equation*}
then from (\ref{q26})
\begin{equation}\label{q3659}
\begin{split}
 \sum_{n=0}^{\infty}\lim_{\alpha\rightarrow\infty} B^{(2)}_{n,\alpha}(x;q)\frac{t^{n}}{[n]_{q}!}
=\sum_{n=0}^{\infty}\frac{q^{n(n-1)/2}t^{n}}{[n]_{q}!}(\frac{-1}{2})^{n}G_{n}(-2x).\end{split}
\end{equation}
Equating  the coefficients of $ t^{n} $ in (\ref{q3659})  gives  (\ref{q10046}).  The proof of (\ref{q10045}) follows directly
from the relation (\ref{q1698}) since
\begin{equation*}
  1=\displaystyle\lim_{\alpha\rightarrow\infty}g^{(2)}_{\alpha}(it;q)= E_{q}(\frac{t}{2})E_{q}(\frac{-t}{2})\displaystyle\lim_{\alpha\rightarrow\infty}g^{(1)}_{\alpha}(it;q).
\end{equation*}
Hence
\begin{equation*}
  \displaystyle\lim_{\alpha\rightarrow\infty}g^{(1)}_{\alpha}(it;q)=e_{q}(\frac{t}{2})e_{q}(\frac{-t}{2}),\,\,|t(1-q)|<2.
\end{equation*}
Therefore, computing the limit in both sides of (\ref{q14}) gives
\begin{equation*}
  \frac {e_{q}(xt)}{e_{q}(\frac{t}{2})}= \sum_{n=0}^{\infty} \lim_{\alpha\rightarrow\infty}\, B^{(1)}_{n,\alpha}(x;q)\frac{t^{n}}{[n]_{q}!}.
\end{equation*}
From the $q$-binomial theorem (see \cite[ Eq.(1.3.2), P. 8]{Gasper}), we have
\begin{equation*}
  \frac {e_{q}(xt)}{e_{q}(\frac{t}{2})}= \frac{(\frac{t}{2}(1-q);q)_{\infty}}{(xt(1-q);q)_{\infty}}=\sum_{n=0}^{\infty} \frac{(\frac{1}{2x};q)_{n}}{(q;q)_{n}}(xt(1-q))^{n},\,\,|xt(1-q)|<1.
\end{equation*}
Hence
\begin{equation}\label{q15363}
\sum_{n=0}^{\infty} \lim_{\alpha\rightarrow\infty}\, B^{(1)}_{n,\alpha}(x;q)\frac{t^{n}}{[n]_{q}!}= \sum_{n=0}^{\infty}\frac{(xt)^{n}}{[n]_{q}!}(\frac{1}{2x};q)_{n},
\end{equation}
 equating the coefficients of $t^{n}$ in (\ref{q15363}) yields  the required result.}
\end{proof}
\begin{corollary}{\small
For $ n\in\mathbb{N}$,
\begin{equation}\label{q6735}
  \lim_{\alpha\rightarrow\infty}\,\beta_{n,\alpha}(q)=(-1)^{n}2^{-n}q^{\frac{n(n-1)}{2}}.
\end{equation}}
\end{corollary}
\begin{proof}{\small
 Since
\begin{equation*}
\begin{split}
  \lim_{x\rightarrow 0}x^{n}(\frac{1}{2x};q)_{n}= \lim_{x\rightarrow 0}x^{n}\prod_{k=0}^{n-1}(1-\frac{q^{k}}{2x})=\lim_{x\rightarrow 0}\prod_{k=0}^{n-1}(x-\frac{q^{k}}{2})=(-1)^{n}2^{-n}q^{\frac{n(n-1)}{2}},\end{split}
\end{equation*}
 then substituting with $x=0$ into (\ref{q10045}) yields (\ref{q6735}).}
\end{proof}

{\small\begin{lemma}\label{dd}
 Let $\alpha_{0}>-1$.  If   $q^{3/2}(1-q)^{2}<(1-q^{2})(1-q^{2\alpha_{0}+2})$,  then  $(q^{\frac{1}{4}}t/2)^{-\alpha}J_{\alpha}^{(3)}(\frac {t}{2}(1-q)q^{\frac{-1}{4}};q^{2})$ has no zeros in $|t|\leq 1$ for all $\alpha\geq\alpha_{0}.$
\end{lemma}}
\begin{proof}
The proof is similar to the proof of Proposition \ref{kk} and is omitted.
\end{proof}
\begin{theorem}{\small
For $ n\in\mathbb{N}$,
\begin{equation}\label{q10047}
 \quad\quad\quad \quad\quad \lim_{\alpha\rightarrow\infty}\,B^{(3)}_{n,\alpha}(x;q)=q^{\frac{n(n-1)}{4}}(\frac{-1}{2})^{n}\sum_{k=0}^{[\frac{n}{2}]}\frac{(-1)^{k}q^{k(n-k+3)} (q^{-n};q)_{2k}} {(q^{2};q^{2})_{k}}(2xq^{\frac{1-n}{2}};q)_{n-2k},
  \end{equation}
\begin{equation}\label{g10047}
 \lim_{\alpha\rightarrow\infty}\,\beta^{(3)}_{n,\alpha}(q)=q^{\frac{n(n-1)}{4}}(\frac{-1}{2})^{n}\sum_{k=0}^{[\frac{n}{2}]}\frac{(-1)^{k}
 q^{k(n-k+3)}(q^{-n};q)_{2k}}{(q^{2};q^{2})_{k}}.
\end{equation}}
\end{theorem}
\begin{proof}{\small
Taking the limit as $\alpha\rightarrow\infty$ on both sides of (\ref{q197}), we obtain
\begin{equation}\label{q3593}
 \displaystyle\lim_{\alpha\rightarrow\infty} \frac{\exp_{q}(xt)\exp_{q}(\frac{-t}{2})}{ g^{(3)}_{\alpha}(it;q)}= \lim_{\alpha\rightarrow\infty}\,\sum_{n=0}^{\infty} B^{(3)}_{n,\alpha}(x;q)\frac{t^{n}}{[n]_{q}!}.
\end{equation}
We can choose $\alpha_{0}>-1$ such that
$$q^{3/2}\leq\frac{(1-q^{2})}{1-q}\frac{(1-q^{2\alpha_{0}+2})}{1-q}\leq\frac{(1-q^{2})}{1-q}\frac{(1-q^{2\alpha+2})}{1-q},$$
for all $\alpha\geq\alpha_{0}$.  Hence from Lemma \ref{dd}, the function $g^{(3)}_{\alpha}(it;q)$ does not vanish on the unit disk, and  the left hand side of (\ref{q3593})  is analytic for $|t|\leq1$. Therefore, we can interchange the limit  as $\alpha\rightarrow\infty$  with the summation in (\ref{q3593}) to obtain
\begin{equation*}
  \frac{\exp_{q}(xt)\exp_{q}(\frac{-t}{2})}{\displaystyle \lim_{\alpha\rightarrow\infty} g^{(3)}_{\alpha}(it;q)}= \sum_{n=0}^{\infty} \lim_{\alpha\rightarrow\infty}\, B^{(3)}_{n,\alpha}(x;q)\frac{t^{n}}{[n]_{q}!}.
\end{equation*}
Since
\begin{equation*}
\begin{split}
 \lim_{\alpha\rightarrow\infty} g^{(3)}_{\alpha}(it;q)&=
 \sum_{n=0}^{\infty}\lim_{\alpha\rightarrow\infty}\frac{q^{n^{2}+\frac{n}{2}}(1-q)^{2n}}{(q^{2},q^{2\alpha+2};q^{2})_{n}}(\frac{t}{2})^{2n} \\& =\sum_{n=0}^{\infty}\frac{q^{n^{2}+\frac{n}{2}}(1-q)^{2n}(\frac{t}{2})^{2n}}{(q^{2};q^{2})_{n}} =(-\frac{q^{\frac{3}{2}}(1-q)^{2}t^{2}}{4};q^{2})_{\infty}.\end{split}
\end{equation*}
Hence
\begin{equation}\label{q4363}
 \frac{ \exp_{q}(xt)\exp_{q}(\frac{-t}{2})}{(-q^{\frac{3}{2}}(1-q)^{2}t^{2}/4;q^{2})_{\infty}}= \sum_{n=0}^{\infty} \lim_{\alpha\rightarrow\infty}\, B^{(3)}_{n,\alpha}(x;q)\frac{t^{n}}{[n]_{q}!}.
\end{equation}
But
{\small\begin{equation*}
\begin{split}
 \exp_{q}(xt)\exp_{q}(\frac{-t}{2})&=\sum_{n=0}^{\infty}\frac{(\frac{-1}{2})^{n}q^{\frac{n(n-1)}{4}}t^{n}}{[n]_{q}!} (2xq^{\frac{1-n}{2}};q)_{n}.
\end{split}
\end{equation*}}

Therefore,
{\small\begin{equation}\label{q2363}
\begin{split}
 \frac{\exp_{q}(xt)\exp_{q}(\frac{-t}{2})}{(-q^{\frac{3}{2}}(1-q)^{2}t^{2}/4;q^{2})_{\infty}}&=\left(\sum_{n=0}^{\infty}\frac{(-1)^{n}q^{\frac{3}{2}n}
 (\frac{(1-q)t}{2})^{2n}}
 {(q^{2};q^{2})_{n}}\right)\left(\sum_{n=0}^{\infty}\frac{q^{\frac{n(n-1)}{4}}(\frac{-t}{2})^{n}}{[n]_{q}!} (2xq^{\frac{1-n}{2}};q)_{n}\right)\\&
 =\sum_{n=0}^{\infty}q^{\frac{n(n-1)}{4}}\left(\frac{-t(1-q)}{2}\right)^{n}\sum_{k=0}^{[\frac{n}{2}]}\frac{(-1)^{k}q^{\frac{3}{2}k}q^{k^{2}-nk+k/2}}
 {(q^{2};q^{2})_{k}(q;q)_{n-2k}}
 (2xq^{\frac{1-n}{2}};q)_{n-2k}\\&=\sum_{n=0}^{\infty}\frac{q^{\frac{n(n-1)}{4}}(\frac{-t}{2})^{n}}{[n]_{q}!}\sum_{k=0}^{[\frac{n}{2}]}\frac{(-1)^{k}
 q^{k(n-k+3)}(q^{-n};q)_{2k}}{(q^{2};q^{2})_{k}} (2xq^{\frac{1-n}{2}};q)_{n-2k}.
\end{split}
\end{equation}}
Substituting from (\ref{q2363}) into (\ref{q4363}) and equating the coefficients of $t^{n} $ yields (\ref{q10047}). The proof of (\ref{g10047})  follows directly by setting $x=0$ in (\ref{q10047}).
}
\end{proof}
{\small\begin{theorem}\label{yy}
Let $\alpha $ be a complex number such that $ Re\,\alpha >-1$. Then for\\ $n\in\mathbb{N},\,n\,\geq\,2$,
\begin{align}\label{q72}
  \beta_{n,\alpha}(q)&=-\frac{[n]_{q}!(1-q)^{2}}{4}\sum_{k=0}^{[\frac{n}{2}]-1}\frac{\left((1-q)/2\right)^{2k}\beta_{n-2k-2,\alpha}(q)}{[n-2k-2]_{q}!(q^{2},q^{2\alpha+2};q^{2})_{k+1}}
  +\frac{(-1)^{n}}{2^{n}},\\
  \beta_{n,\alpha}^{(3)}(q)&=-\frac {[n]_{q}!q^{3/2}(1-q)^{2}}{4}\sum_{k=0}^{[\frac{n}{2}]-1}\frac{q^{k^{2}+5k/2}\left((1-q)/2\right)^{2k}\beta^{(3)}_{n-2k-2,\alpha}(q)}
  {[n-2k-2]_{q}!(q^{2},q^{2\alpha+2};q^{2})_{k+1}}+q^{\frac{n(n-1)}{4}}\frac{(-1)^{n}}{2^{n}}.
\end{align}
\end{theorem}}
\begin{proof}{\small
We give in detail the proof of (\ref{q72}). The proof for $\beta_{n,\alpha}^{(3)}(q)$ is similar. Since
\begin{equation}\label{q149}
  \frac{e_{q}(\frac{-t}{2})}{g^{(1)}_{\alpha}(it;q)}=\sum_{n=0}^{\infty}\beta_{n,\alpha}(q)\frac{t^{n}}{[n]_{q}!},
\end{equation}
then
\begin{equation*}
\frac{e_{q}(\frac{-t}{2})}{g^{(1)}_{\alpha}(it;q)}-e_{q} (\frac{-t}{2}) =\sum_{n=0}^{\infty}\beta_{n,\alpha}(q)\frac{t^{n}}{[n]_{q}!}- e_{q} (\frac{-t}{2}).
\end{equation*}
Consequently, from the series representation of $e_{q}(t)$ in (\ref{h8}), we get
\begin{equation}\label{q5002} \frac{e_{q}(\frac{-t}{2})}{g^{(1)}_{\alpha}(it;q)}\left(1-g^{(1)}_{\alpha}(it;q)\right)=\sum_{n=0}^{\infty}\left(\beta_{n,\alpha}(q)
-\frac{(-1)^{n}}{2^{n}}\right)\frac{t^{n}}{[n]_{q}!}.
\end{equation}
Since
\begin{equation}\label{q54}\begin{split}
 \left(g^{(1)}_{\alpha}(it;q)-1\right)=t^{2} \sum_{m=0}^{\infty}\frac{(1-q)^{2m+2}t^{2m}}{2^{2m+2}(q^{2},q^{2\alpha+2};q^{2})_{m+1}},\end{split}
\end{equation}
then substituting from (\ref{q54}) into (\ref{q5002}) and using (\ref{q149}),  we obtain
\begin{equation*}
\begin{split}
  \left( \sum_{n=0}^{\infty}\beta_{n,\alpha}(q)\frac{t^{n}}{[n]_{q}!}\right)\left(-t^{2} \sum_{n=0}^{\infty}\frac{(1-q)^{2n+2}t^{2n}}{2^{2n+2}(q^{2},q^{2\alpha+2};q^{2})_{n+1}}\right)=\sum_{n=0}^{\infty}\left(\beta_{n,\alpha}(q)-
\frac{(-1)^{n}}{2^{n}}\right)\frac{t^{n}}{[n]_{q}!}.\end{split}
\end{equation*}
Therefore, by the Cauchy product formula
\begin{equation}\label{q59}
-\frac{(1-q)^{2}}{4}\sum_{n=2}^{\infty}t^{n}\sum_{k=0}^{[\frac{n}{2}]-1}\frac{(1-q)^{2k}\beta_{n-2k-2,\alpha}(q)}{2^{2k}[n-2k-2]_{q}!(q^{2},q^{2\alpha+2};q^{2})_{k+1}}
=\sum_{n=0}^{\infty}\left(\beta_{n,\alpha}(q)-\frac{(-1)^{n}}{2^{n}}\right)\frac{t^{n}}{[n]_{q}!}.
\end{equation}
Equating  the coefficient of $ t^{n}$ in (\ref{q59}), we get (\ref{q72}) and the theorem follows.}
\end{proof}
\quad
\quad
{\small The following theorem gives a recursive relations between the polynomials $B^{(k)}_{n,\alpha}(x;q)$ and  $B^{(k)}_{n,\alpha+1}(x;q)\,(k=2,3)$.}
\begin{theorem}{\small
If $ Re\,\alpha > -1 $ , $ x\in \mathbb{C},$  and $ k \in\,\mathbb{N}$,  then
\begin{equation*}
\frac {B^{(r)}_{n,\alpha}(x;q)}{[n]_{q}!}=2(1-q^{2\alpha+2})\sum_{k=0}^{[\frac{n}{2}]}(-1)^{k}\frac{(1-q)^{2k}\,h^{(r)}_{k+1}(q^{2})\,}
{[n-2k]_{q}!}B^{(r)}_{n-2k,\alpha+1}(x;q)\quad (r=2,3),
\end{equation*}
where
\begin{equation*}
  h^{(r)}_{k}(q^{2})=\sum_{m=1}^{\infty}\frac{-2J^{(r)}_{\alpha+1}(j^{(r)}_{m,\alpha};q^{2})}{ \frac{d}{dz}J^{(r)}_{\alpha}(z;q^{2})|_{z=j^{(r)}_{m,\alpha}}}\left(\frac{1}{j^{(r)}_{m,\alpha}}\right)^{2k},
\end{equation*}
and $(j^{(r)}_{m,\alpha})_{m=1}^{\infty}\,\,(r=2,3)$  are the positive zero of $ J^{(r)}_{\alpha}( \cdots;q^{2}).$}
\end{theorem}
\begin{proof}{\small
We start with  the proof of the case $(r=2)$. From \cite{ZMA,AZ}, we have the identity
\begin{equation}\label{q8219}
  \frac{J^{(2)}_{\alpha+1}(t;q)}{J^{(2)}_{\alpha}(t;q)}=\sum_{n=1}^{\infty}h^{(2)}_{n}(q)t^{2n-1},
\end{equation}
where
\begin{equation*}
  h^{(2)}_{n}(q)=\sum_{m=1}^{\infty}\frac{-2J^{(2)}_{\alpha+1}(j^{(2)}_{m,\alpha};q^{2})}{ \frac{d}{dz}J^{(2)}_{\alpha}(z;q^{2})|_{z=j^{(2)}_{m,\alpha}} }\left(\frac{1}{j^{(2)}_{m,\alpha}}\right)^{2n}.
\end{equation*}
Replacing $ t$ by $ it(1-q)$ and $ q$ by $ q^{2}$  in (\ref{q8219}),  we obtain
\begin{equation}\label{q210}
   \frac{1}{J^{(2)}_{\alpha}(it(1-q);q^{2})}= \frac{1}{J^{(2)}_{\alpha+1}(it(1-q);q^{2})}\sum_{n=1}^{\infty}h^{(2)}_{n}(it(1-q))^{2n-1}.
\end{equation}
Multiplying (\ref{q210}) by $ E_{q}(xt) E_{q}(\frac{-t}{2})$ to obtain
\begin{equation}\label{q211}
  \frac{E_{q}(xt) E_{q}(\frac{-t}{2})}{J^{(2)}_{\alpha}(it(1-q);q^{2})}= \frac{E_{q}(xt) E_{q}(\frac{-t}{2})}{J^{(2)}_{\alpha  +1}(it(1-q);q^{2})}\sum_{n=1}^{\infty}h^{(2)}_{n}(q^{2})(it(1-q))^{2n-1}.
\end{equation}
Substituting from (\ref{q8432}) into (\ref{q211}),  we get
\begin{equation}\label{q213}
\begin{split}
  \frac{E_{q}(xt) E_{q}(\frac{-t}{2})}{g^{(2)}_{\alpha}(it;q)}=\frac{(1+q)[\alpha+1]_{q^{2}}}{(\frac{it}{2})}\frac{E_{q}(xt) E_{q}(\frac{-t}{2})}{ g^{(2)}_{\alpha+1}(it;q)}\sum_{n=1}^{\infty}h^{(2)}_{n}(q^{2})(it(1-q))^{2n-1}.
 \end{split}
\end{equation}
Consequently,
\begin{equation*}
\begin{split}
  \sum_{n=0}^{\infty}B^{(2)}_{n,\alpha}(x;q)\frac{t^{n}}{[n]_{q}!}&=\frac{2(1+q)}{it}[\alpha+1]_{q^{2}}\left(\sum_{n=0}^{\infty}B^{(2)}_{n,\alpha+1}(x;q)
  \frac{t^{n}}{[n]_{q}!}
  \right)\left(\sum_{n=1}^{\infty}h^{(2)}_{n}(q^{2})(it(1-q))^{2n-1}\right)\\&=2(1-q^{2\alpha+2})
 \left(  \sum_{n=0}^{\infty}B^{(2)}_{n,\alpha+1}(x;q)\frac{t^{n}}{[n]_{q}!}\right) \left( \sum_{n=0}^{\infty}h_{n+1}(q^{2})t^{2n}(i(1-q))^{2n} \right) \\&=2(1-q^{2\alpha+2})\sum_{n=0}^{\infty}t^{n}\sum_{k=0}^{[\frac{n}{2}]}(-1)^{k}
  \frac{(1-q)^{2k}
 \,h^{(2)}_{k+1}(q^{2})}{[n-2k]_{q}!}B^{(2)}_{n-2k,\alpha+1}(x;q).
  \end{split}
\end{equation*}
 Hence
\begin{equation}\label{q215}
 \sum_{n=0}^{\infty}B^{(2)}_{n,\alpha}(x;q)\frac{t^{n}}{[n]_{q}!}={2(1-q^{2\alpha+2})}\sum_{n=0}^{\infty}t^{n}\sum_{k=0}^{[\frac{n}{2}]}(-1)^{k}\frac{(1-q)^{2k}
 \,h^{(2)}_{k+1}(q^{2})\,}{[n-2k]_{q}!}B^{(2)}_{n-2k,\alpha+1}(x;q).
\end{equation}
 Equating the coefficients of $ t^{n}$  in (\ref{q215}), we get the result for $(r=2)$. The proof of the case $(r=3)$ follows from the identity (see \cite[ Eq. (4.3), P. 6]{abr}),
 \begin{equation*}
\begin{split}
  \frac{J^{(3)}_{\alpha+1}(t;q)}{J^{(3)}_{\alpha}(t;q)}=\sum_{n=1}^{\infty}h^{(3)}_{n}(q)t^{2n-1},
  \end{split}
\end{equation*}
where
\begin{equation*}
  h^{(3)}_{n}(q)=\sum_{m=1}^{\infty}\frac{-2J^{(3)}_{\alpha+1}(j^{(3)}_{m,\alpha};q^{2})}{ \frac{d}{dz}J^{(3)}_{\alpha}(z;q^{2})|_{z=j^{(3)}_{m,\alpha}} }\left(\frac{1}{j^{(3)}_{m,\alpha}}\right)^{2n},
\end{equation*}
 and by using the same technique.}
\end{proof}

\section{Asymptotic relations for the generalized $q$-Bernoulli numbers}
In this section,  we derive asymptotic relations for the generalized $q$-Bernoulli numbers defined in (\ref{q66}).
\begin{theorem}\label{q9700}{\small
Let $n$ be a non negative integer and $\alpha$ be a complex number such that\\ $Re\,\alpha>-1$. Then for $n\in\mathbb{N},$
\begin{align}\label{s1112}
\begin{split}
\beta_{2n,\alpha}(q)&=2(-1)^{n+1}(q;q)_{2n}
\sum_{k=1}^{\infty}\frac{Cos_{q}(\frac{j^{(2)}_{k,\alpha}}{2(1-q)})}{(j^{(2)}_{k,\alpha})^{2n+1}\,
\frac{d}{dz}\mathcal{J}_{\alpha}^{(2)}(z;q^{2})|_{z=j^{(2)}_{k,\alpha}}},\\
\beta_{2n+1,\alpha}(q)&=2(-1)^{n}(q;q)_{2n+1}
\sum_{k=1}^{\infty}\frac{Sin_{q}(\frac{j^{(2)}_{k,\alpha}}{2(1-q)})}
{(j^{(2)}_{k,\alpha})^{2n+2}\,\frac{d}{dz}\mathcal{J}_{\alpha}^{(2)}(z;q^{2})|_{z=j^{(2)}_{k,\alpha}}},
\end{split}
\end{align}
where $\mathcal{J}_{\alpha}^{(2)} (z;q)$ is defined in (\ref{g7907}).
}
\end{theorem}
\begin{proof}{\small
Since
\begin{equation*}
  G(z):=\frac{E_{q}(\frac{-z}{2})}{g^{(2)}_{\alpha}(iz;q)}=\sum_{n=0}^{\infty}\beta_{n,\alpha}(q)\frac{z^{n}}{[n]_{q}!},\,\,\,|z|<\frac{j^{(2)}_{1,\alpha}}{1-q},
\end{equation*}
then
\begin{equation*}
 \frac{\beta_{n,\alpha}(q)}{[n]_{q}!}=\frac{G^{(n)}(0)}{n!},\quad n\in\mathbb{N}_{0}.
\end{equation*}
Now,  we integrate $f(z):=\dfrac{G(z)}{z^{n+1}},\,G(z)=\dfrac{E_{q}(\frac{-z}{2})}{g^{(2)}_{\alpha}(iz;q)}$ on the contour $\Gamma_{m},$ where
$\Gamma_{m}$  is a circle of radius $R_{m},$  $|z_{m}|<\,R_{m}<|z_{m+1}|$. From the Cauchy Residue Theorem, see \cite{cauchy},
\begin{equation*}
\int_{\Gamma_{m}} f(z) \,dz=2\pi i\sum\,Res(f,z_{k} ),
\end{equation*}
where $\{z_{k}\}$ are the poles of $f$ that lie inside $\Gamma_{m}$. The function $f(z)$ has a pole at $z=0$ of order $n+1$ and simple poles at  $\pm z_{k}$ where $z_{k}= i\dfrac{j^{(2)}_{k,\alpha}}{1-q},\,k\in\mathbb{N}$. Consequently,
\begin{equation}\label{s403}
I_{m}=\frac{1}{2\pi i}\int_{\Gamma_{m}} f(z) \,dz=Res(f(z),0)+\sum_{k=1}^{m}Res(f(z),\pm z_{k} ).
\end{equation}
  Since
{\small\begin{equation*}
Res(f,0)=\frac{f^{n}(0)}{n!}=\frac{\beta_{n,\alpha}(q)}{[n]_{q}!},
\end{equation*}}
\begin{equation*}
\begin{split}
Res(f,z_{k} )=\frac{E_{q}(\frac{-z_{k}}{2})}{\frac{d}{dz}g^{(2)}_{\alpha}(iz;q)|_{z=z_{k}}}\frac{1}{(z_{k})^{n+1}}=\frac{E_{q}(\frac{-ij^{(2)}_{k,\alpha}}{2(1-q)})}
{\frac{d}{dz}\mathcal{J}^{(2)}_{\alpha}(z;q^{2})|_{z=j^{(2)}_{k,\alpha}}}\frac{(i)^{-n}(1-q)^{n}}{(j^{(2)}_{k,\alpha})^{n+1}},
 \end{split}
\end{equation*}
and
\begin{equation*}
\begin{split}
Res(f,-z_{k} )=\frac{E_{q}(\frac{z_{k}}{2})}{\frac{d}{dz}g^{(2)}_{\alpha}(iz;q)|_{z=-z_{k}}}\frac{1}{(-z_{k})^{n+1}}=\frac{E_{q}(\frac{ij^{(2)}_{k,\alpha}}{2(1-q)})}
{\frac{d}{dz}\mathcal{J}^{(2)}_{\alpha}(z;q^{2})|_{z=j^{(2)}_{k,\alpha}}}\frac{(-i)^{-n}(1-q)^{n}}{(j^{(2)}_{k,\alpha})^{n+1}}.
 \end{split}
\end{equation*}
Then Equation (\ref{s403}) can be written as
\begin{equation}\label{s12}
\begin{split}
I_{m}&=
\frac{\beta_{n,\alpha}(q)}{[n]_{q}!}+\sum_{k=1}^{m}2Re\,\left((-i)^{-n}E_{q}(\frac{ij^{(2)}_{k,\alpha}}{2(1-q)})\right)\frac{(1-q)^{n}}
{(j^{(2)}_{k,\alpha})^{n+1}\frac{d}{dz}\mathcal{J}^{(2)}_{\alpha}(z;q^{2})|_{z=j^{(2)}_{k,\alpha}}},
 \end{split}
\end{equation}
substituting into (\ref{s12}) with $-i=e^{-\frac{i\pi}{2}}$ gives
\begin{equation*}
\begin{split}
I_{m}&=
\frac{\beta_{n,\alpha}(q)}{[n]_{q}!}+2(1-q)^{n}\cos\frac{n\pi}{2}
\sum_{k=1}^{m}\frac{Cos_{q}(\frac{j^{(2)}_{k,\alpha}}{2(1-q)})}{\frac{d}{dz}\mathcal{J}^{(2)}_{\alpha}(z;q^{2})|_{z=j^{(2)}_{k,\alpha}}}\frac{1}{(j^{(2)}_{k,\alpha})^{n+1}}
\\&-2(1-q)^{n}\sin\frac{n\pi}{2}\sum_{k=1}^{m}\frac{Sin_{q}(\frac{j^{(2)}_{k,\alpha}}{2(1-q)})}{\frac{d}{dz}\mathcal{J}^{(2)}_{\alpha}(z;q^{2})|_{z=j^{(2)}_{k,\alpha}}}\frac{1}{(j^{(2)}_{k,\alpha})^{n+1}}.
 \end{split}
\end{equation*}
\\
Now, we show that the integral $I_{m}\rightarrow 0$  as $m\rightarrow\infty$.  Bergweiller and Hayman \cite{HB} introduced the asymptotic relation for $E_{q}(z)$,
\begin{equation*}
|M(r;E_{q})|:=\sup\{|E_{q}(z)|:|z|=r\}\sim \,e^{\frac{-(\log r)^{2}}{2\log q}},\quad when \quad r=|z|\rightarrow\infty.
\end{equation*}
In \cite{Annaby1}, Annaby and Mansour proved that for $r=|z|\,\rightarrow \,\infty$

\begin{equation*}
 z^{-\nu}J_{\nu}^{(2)}(z;q) \sim \,\exp\left(-\frac{(\log\,r)^{2}}{2\log\,q}-\frac{\log\,2}{\log\,q}\log\,r\right).
\end{equation*}
\vskip 0.2cm
Hayman in \cite{haymen} introduced the higher order asymptotics of $J_{\nu}^{(2)}(z;q)$. Then,  Annaby and Mansour, see \cite{Annaby1},  pointed out that the first order asymptotics  of the zeros of  $J_{\nu}^{(2)}(z;q^{2})$ is given by
\begin{equation*}
  j_{m,\nu}^{(2)}=2q^{-2m}q^{-\nu+1}(1+O(q^{2m})), \quad (m\rightarrow\infty).
\end{equation*}
Hence if $(z_{m})_{m}$ are the positive zero of $g^{(2)}_{\alpha}(iz;q)$,  then
\begin{equation}\label{k6}
  \lim_{m\rightarrow\infty}\frac{z_{m}}{z_{m+1}}=\lim_{m\rightarrow\infty}\frac{j^{(2)}_{m,\nu}}{j^{(2)}_{m+1,\nu}}=q^{2}, \quad \lim_{m\rightarrow\infty}z_{m}=\infty.
\end{equation}
Let $ {\small 0<\epsilon<(q^{-1}-1)}$. There exists $M_{0}\in\mathbb{N}$ such that if $m\in\mathbb{N},\,m\geq M_{0}$, then
\begin{equation*}\label{w2}
  q^{2}(1-\epsilon)<\frac{z_{m}}{z_{m+1}}<q^{2}(1+\epsilon).
\end{equation*}
 Hence $z_{m}<qz_{m+1}$ for all $m\geq M_{0}$.  We can  choose $R_{m},\,\delta:=q^{-1}\displaystyle\sup_{m\geq M_{0}}\frac{z_{m}}{z_{m+1}}$ such that $(z_{m}<\delta R_{m}<q z_{m+1} <R_{m})$. Indeed,
\begin{equation*}
  \delta=q^{-1}\displaystyle \sup_{m\geq M_{0}}\frac{z_{m}}{z_{m+1}}\geq q^{-1}\frac{z_{m}}{z_{m+1}},\,\, m\geq M_{0}.
\end{equation*}
But $q z_{m+1} <R_{m}$ leads to  $\delta>\frac{z_{m}}{R_{m}}$ and so $ z_{m} <\delta R_{m}$. Now,
\begin{equation*}
  \delta=q^{-1}\displaystyle \sup_{ m\geq M_{0}}\frac{z_{m}}{z_{m+1}}\geq q^{-1}\displaystyle\lim_{m\rightarrow\infty}\frac{z_{m}}{z_{m+1}}=q^{-1}q^{2}=q.
\end{equation*}
Also  $\delta =q^{-1}\displaystyle \sup_{ m\geq M_{0}}\frac{z_{m}}{z_{m+1}}<q(1+\epsilon)<1$.  Hence $1>\delta>q$ and so by
\begin{equation}\label{k7}
  z_{m}< R_{m}<\frac{q}{\delta} z_{m+1}<z_{m+1},
\end{equation}
the annulus $\delta R_{m}<|z| <R_{m}$ has no zeros of the function $ g^{(2)}_{\alpha}(iz;q)$. Hence, from the minimum modulus principle we have
\begin{align}\label{s98}
\begin{split}
  \left|g^{(2)}_{\alpha}(iz;q)\right|&\geq c_{1}\,e^{-\frac{(\log \delta R_{m})^{2}}{2\log q}-\frac{\log2}{\log q}\log\delta R_{m}},\quad c_{1}>0.\\
   \left|E_{q}(\frac {-z}{2}) \right|&\leq \,c_{2}\,e^{\frac{-(\log \frac{R_{m}}{2})^{2}}{2\log q}},\quad c_{2}>0.\end{split}
\end{align}
Therefore, from (\ref{s98}), we conclude that

\begin{equation*}
\begin{split}
  \left|\frac{E_{q}(\frac{-z}{2})}{g^{(2)}_{\alpha}(iz;q)} \right|&\leq \,\frac{c_{2}}{c_{1}}\,\,\frac {e^{\frac{-(\log\frac{ R_{m}}{2})^{2}}{2\log q}}}
{e^{-\frac{( \log\delta R_{m})^{2}}{2\log q}-\frac{\log2}{\log q}\log\delta R_{m}}}\\& \leq \,\frac{ c_{2}}{c_{1}}\,e^{\frac{1}{2\log q}\left(((\log\delta R_{m})^{2}-(\log\frac {R_{m}}{2})^{2})\right)+\frac{\log2}{\log q}\log\delta R_{m}} \\&\leq\,\frac{ c_{2}}{c_{1}}\,e^{K}\,e^{\frac{2\log2 \log R_{m}}{\log q}+\frac{\log\delta \log R_{m}}{\log q}},
 \end{split}
\end{equation*}
where
\begin{equation*}
  K=\frac{1}{2\log q}\left((\log\delta)^{2} -(\log2)^{2}+2\log2\log\delta\right).
\end{equation*}
Now, using the ML-inequality (see\cite{cauchy}) to obtain
\begin{equation}\label{s43}
\begin{split}
|I_{m}|&=|\int_{\Gamma_{m}}f(z)dz|\leq(2\pi R_{m})|M(r;f(z))|\\&\leq\,\frac{2\pi R_{m}c_{2}}{c_{1}}\,e^{K}\,e^{\frac{2\log2 \log R_{m}}{\log q}+\frac{\log\delta \log R_{m}}{\log q}}\frac{1}{R_{m}^{n+1}}\\&\leq\,\frac{ 2\pi c_{2}}{c_{1}}\,e^{K}\,R_{m}^{\frac{2 \log2}{\log q}}R_{m}^{\frac{\log\delta }{\log q}-n}.
\end{split}
\end{equation}
 From (\ref{k6}) and (\ref{k7}), we have $\displaystyle\lim_{m\rightarrow\infty}R_{m}=0$. Also, since $0<q<1$  and  $1>\delta>q$ then
\begin{equation*}
  R_{m}^{\frac{2 \log2}{\log q}}\rightarrow 0\,\,\,\, and \,\,\,\,R_{m}^{\frac{\log\delta }{\log q}-n}\rightarrow 0 \quad as\,\, m\,\,\rightarrow\,\infty.
\end{equation*}

Hence $\displaystyle\lim_{m\rightarrow\infty}I_{m}=0.$   Consequently,
\begin{equation*}
\begin{split}
\frac{\beta_{n,\alpha}(q)}{[n]_{q}!}&=-2(1-q)^{n}\cos\frac{n\pi}{2}
\sum_{k=1}^{\infty}\frac{Cos_{q}(\frac{j^{(2)}_{k,\alpha}}{2(1-q)})}{\frac{d}{dz}\mathcal{J}^{(2)}_{\alpha}(z;q^{2})|_{z=j^{(2)}_{k,\alpha}}}\frac{1}{(j^{(2)}_{k,\alpha})^{n+1}}
\\&+2(1-q)^{n}\sin\frac{n\pi}{2}\sum_{k=1}^{\infty}\frac{Sin_{q}(\frac{j^{(2)}_{k,\alpha}}{2(1-q)})}{\frac{d}{dz}\mathcal{J}^{(2)}_{\alpha}(z;q)|_{z=j^{(2)}_{k,\alpha}}}\frac{1}{(j^{(2)}_{k,\alpha})^{n+1}}.
 \end{split}
\end{equation*}
Therefore,
\begin{align*}
\begin{split}
\beta_{2n,\alpha}(q)&=2(-1)^{n+1}(q;q)_{2n}
\sum_{k=1}^{\infty}\frac{Cos_{q}(\frac{j^{(2)}_{k,\alpha}}{2(1-q)})}{(j^{(2)}_{k,\alpha})^{2n+1}\,
\frac{d}{dz}\mathcal{J}_{\alpha}^{(2)}(z;q^{2})|_{z=j^{(2)}_{k,\alpha}}},\\
\beta_{2n+1,\alpha}(q)&=2(-1)^{n}(q;q)_{2n+1}
\sum_{k=1}^{\infty}\frac{Sin_{q}(\frac{j^{(2)}_{k,\alpha}}{2(1-q)})}
{(j^{(2)}_{k,\alpha})^{2n+2}\,\frac{d}{dz}\mathcal{J}_{\alpha}^{(2)}(z;q^{2})|_{z=j^{(2)}_{k,\alpha}}},
\end{split}
\end{align*}
which completes the proof of the theorem.
}
\end{proof}
\begin{remark}{\small
If we substitute with $\alpha=\frac{1}{2}$ in the second equation in (\ref{s1112}), then $(z_{k})_{k}$ will be the positive zeros of $Sin_{q}(z)$ and consequently, the series in the left hand side vanishes which coincide with the known result that the odd Bernoulli numbers vanish $(\beta_{2n+1}(q)=0,\,n\geq\, 1)$ (see \cite{IZ}). Similarly, if we set $\alpha=-\frac{1}{2}$ in the first equation in (\ref{s1112}), the series in the left hand side vanishes  and this coincide with the fact that the even Euler$^{,}$s numbers are zero  $(E_{2n}(q)=0,\,n\geq\,1)$ (see \cite{IZ}).}
\end{remark}
\begin{corollary}{\small
The asymptotic relations of the generalized $q$-Bernoulli numbers $(\beta_{n,\alpha}(q))_{n},$
\begin{align*}
\begin{split}
\beta_{2n,\alpha}(q)&=2(-1)^{n+1}(q;q)_{2n}
\frac{Cos_{q}(\frac{j^{(2)}_{1,\alpha}}{2(1-q)})}{(j^{(2)}_{1,\alpha})^{2n+1}\,
\frac{d}{dz}\mathcal{J}_{\alpha}^{(2)}(z;q^{2})|_{z=j^{(2)}_{1,\alpha}}}\left(1+o(1)\right),\\
\beta_{2n+1,\alpha}(q)&=2(-1)^{n}(q;q)_{2n+1}
\frac{Sin_{q}(\frac{j^{(2)}_{1,\alpha}}{2(1-q)})}
{(j^{(2)}_{1,\alpha})^{2n+2}\,\frac{d}{dz}\mathcal{J}_{\alpha}^{(2)}(z;q^{2})|_{z=j^{(2)}_{1,\alpha}}}\left(1+o(1)\right),
\end{split}
\end{align*}
where $\mathcal{J}_{\alpha}^{(2)} (z;q)$ is defined in (\ref{g7907}).
}
\end{corollary}
\begin{proof}{\small
The proof follows directly from Theorem \ref{q9700}.}
\end{proof}
% ----------------------------------------------------------------

\section{ Applications of the generalized $q$-Bernoulli polynomials} \label{Applications}
In this section, we introduce connection relations between the generalized $q$-Bernoulli polynomials ${\small B^{(k)}_{n,\alpha}(x;q)\,(k=1,2,3)}$ and the $q$-Laguerre  and  the little $q$-Legendre polynomials.\\
\\
{\small
The $q$-Laguerre polynomials $ L_{n}^{\alpha}(x;q)$ of degree $n$  are defined by
\begin{equation}\label{L1}
\begin{split}
 L_{n}^{\alpha}(x;q):&=\frac{1}{(q;q)_{n}}\,_{2}\varphi_{1} \left(
                                                                               \begin{array}{c}
                                                                                 q^{-n},-x \\
                                                                                 0 \\
                                                                               \end{array}
                                                                           q;q^{n+\alpha+1}  \right)
                                                                           \\&=\frac{(q^{\alpha+1};q)_{n}}{(q;q)_{n}}\sum_{k=0}^{n}
                                                                           \frac{(q^{-n};q)_{k}}{(q^{\alpha+1};q)_{k}}(-1)^{k}(q^{n+\alpha+1})^{k}x^{k}.
                                                                           \end{split}
\end{equation}}
{\small The Rodrigues formula is given by
\begin{equation}\label{L2}
\begin{split}
  L_{n}^{\alpha}(x;q)=\frac{(1-q)^{n}}{(q;q)_{n}}(-x;q)_{\infty}\,x^{-\alpha}\,D_{q}^{n}\left(\frac{x^{\alpha+n}}{(-x;q)_{\infty}}\right),
 \end{split}
\end{equation}}
{\small and the orthogonality relation is
\begin{equation}\label{L3}
  \int_{0}^{\infty}\frac{x^{\alpha}}{(-x;q)_{\infty}}L_{m}^{\alpha}(x;q)L_{n}^{\alpha}(x;q)dx
  =\frac{(q^{-\alpha};q)_{\infty}}{(q;q)_{\infty}}\frac{(q^{\alpha+1};q)_{n}}{(q;q)_{n}q^{n}}\Gamma_{q}(-\alpha)\Gamma_{q}(\alpha+1)\delta_{mn},\quad \alpha>-1,
\end{equation}
 where $\delta_{mn}$ is the Kronecker delta function, see \cite{Asky, ismail}}. {\small The $q$-Laguerre polynomials $ L_{n}^{\alpha}(x;q)$ satisfy three term recurrence relation}
{\small\begin{equation*}
\begin{split}
 -xa_{n}L_{n}^{\alpha}(x;q)=L_{n+1}^{\alpha}(x;q)-b_{n}L_{n}^{\alpha}(x;q)+d_{n}L_{n-1}^{\alpha}(x;q), \end{split}
\end{equation*}}
where
\begin{equation*}
  a_{n}=\frac{q^{2n+\alpha+1}}{1-q^{n+1}},\quad b_{n}=1+\frac{q(1-q^{n+\alpha}}{1-q^{n+1}},\quad d_{n}=\frac{q(1-q^{n+\alpha})}{1-q^{n+1}}.
\end{equation*}
{\small In the following, let $\alpha>-1$ and $\mathbb{P}_{n}=\{p(x):\text {deg}\, p(x)\leq n\}$ with the inner product
\begin{equation*}
  \langle p(x),g(x)\rangle= \int_{0}^{\infty}\frac{x^{\alpha}}{(-x;q)_{\infty}}p(x)g(x)dx,
\end{equation*}
where $p(x),\,g(x)\in\mathbb{P}_{n}$. From (\ref{L3}), we note that  $\{L_{0}^{\alpha}(x;q),L^{\alpha}_{1}(x;q),\ldots L^{\alpha}_{n}(x;q)\}$
 is an orthogonal basis for $\mathbb{P}_{n}$.}
\begin{theorem}\label{qq}{\small
Let $p(x)$ $\in\mathbb{P}_{n}.$  Then $p(x)$ can be expanded as
\begin{equation*}
p(x)=\sum_{m=0}^{n}C_{m}L_{m}^{\alpha}(x;q),
\end{equation*}
where
\begin{equation*}
  C_{m}=\frac{q^{m}(1-q)^{m-1}(q^{\alpha+m+1};q)_{\infty}}{(q;q)_{\infty}}\int_{0}^{\infty}
 \,D_{q}^{m}\left(\frac{x^{\alpha+m}}{(-x;q)_{\infty} }\right)p(x) dx.
\end{equation*}}
\end{theorem}
\begin{proof}{\small
Since
\begin{equation*}
p(x)=\sum_{m=0}^{n}C_{m}L_{m}^{\alpha}(x;q),
\end{equation*}
in order to calculate the constant $C_{m}$, we use (\ref{L3}) to obtain
\begin{equation*}
  \left\langle p(x),L_{k}^{\alpha}(x;q)\right\rangle =\langle \sum_{m=0}^{n}C_{m}L_{m}^{\alpha}(x;q),L_{k}^{\alpha}(x;q)\rangle= \sum_{m=0}^{n}C_{m}\langle L_{m}^{\alpha}(x;q),L_{k}^{\alpha}(x;q)\rangle.
\end{equation*}
Then
\begin{equation*}
\begin{split}
\langle p(x),L_{m}^{\alpha}(x;q)\rangle=C_{m}\langle L_{m}^{\alpha}(x;q) ,L_{m}^{\alpha}(x;q)\rangle
=C_{m}\frac{(q^{\alpha+1};q)_{m}}{q^{m}(q;q)_{m}}(1-q)^{1+\alpha}\Gamma_{q}(\alpha+1).
\end{split}
\end{equation*}

Therefore,
\begin{equation}\label{L565}
\begin{split}
  C_{m}= \frac{q^{m}(q;q)_{m}}{(q^{\alpha+1};q)_{m}(1-q)^{1+\alpha}\Gamma_{q}(\alpha+1)}\int_{0}^{\infty}\frac{x^{\alpha}}{(-x;q)_{\infty}}L_{m}^{\alpha}(x;q) p(x) dx.
  \end{split}
\end{equation}
Using (\ref{L2}) with $n$ replaced by $m$, we obtain
\begin{equation*}
\begin{split}
  C_{m}=\frac{q^{m}(1-q)^{m-1}(q^{\alpha+m+1};q)_{\infty}}{(q;q)_{\infty}}\int_{0}^{\infty}
 \,D_{q}^{m}\left(\frac{x^{\alpha+m}}{(-x;q)_{\infty} }\right)p(x) dx,
  \end{split}
\end{equation*}
and the theorem follows.}
\end{proof}
{\small The following Lemma, see  \cite{Jase}, is essential in the proof of Theorem \ref{mm}.
\begin{lemma}\label{tt}
Let the functions $f$ and $g$ be defined and continuous on $[0,\infty]$. Assume that the improper Riemann integrals of the functions $f(x)g(x)$ and $f(x/q)g(x)$ exist on $[0,\infty]$. Then
\begin{equation*}
\begin{split}
  \int_{0}^{\infty}f(x)D_{q}g(x)dx &=\frac{f(0)g(0)}{1-q}\ln q-\frac{1}{q}\int_{0}^{\infty}g(x)D_{q^{-1}}f(x)dx\\&=\frac{f(0)g(0)}{1-q}\ln q-\int_{0}^{\infty}g(qx)D_{q}f(x)dx.
  \end{split}
\end{equation*}
\end{lemma}}
\begin{theorem}\label{mm}{\small
If $n\in\mathbb{N}$ and $x\in\mathbb{C},$ then
\begin{align}\label{L1006}
  B^{(1)}_{n,\alpha}(x;q)\nonumber&=\sum_{m=0}^{n}A_{m}\left(\sum_{k=m}^{n}q^{\frac{k(2n-k+1)}{2}}\frac{ (q^{-n};q)_{k}  (q^{-k};q)_{m}(q^{-\alpha-k};q)_{k}}{(q;q)_{k}}
 \beta_{n-k,\alpha} (q)\right)L_{m} ^{\alpha}(x;q),\\
   B^{(2)}_{n,\alpha}(x;q)\nonumber&=\sum_{m=0}^{n}A_{m}\left(\sum_{k=m}^{n}\frac{q^{nk}(q^{-n};q)_{k}  (q^{-k};q)_{m}(q^{-\alpha-k};q)_{k}}{(q;q)_{k}}
 \beta_{n-k,\alpha} (q)\right)L_{m} ^{\alpha}(x;q),\\
 B^{(3)}_{n,\alpha}(x;q)\nonumber&=\sum_{m=0}^{n}A_{m}\left(\sum_{k=m}^{n}q^{\frac{k(4n-k+1)}{4}}\frac{ (q^{-n};q)_{k}  (q^{-k};q)_{m}(q^{-\alpha-k};q)_{k}}{(q;q)_{k}}
 \beta^{(3)}_{n-k,\alpha} (q)\right)L_{m} ^{\alpha}(x;q),
\end{align}
where
\begin{equation*}
  A_{m}=\frac{-q^{m}(q^{\alpha+m+1},q^{-\alpha};q)_{\infty}}{(1-q)^{2}(q,q;q)_{\infty}}\frac{\pi}{sin(\alpha\pi)}.
\end{equation*}}
\end{theorem}
\begin{proof}{\small
We prove the identity  for $B^{(1)}_{n,\alpha}(x;q)$ and the proofs for $B^{(k)}_{n,\alpha}(x;q)\,(k=2,3)$ are similar.  Substitute with
$p(x)=B^{(1)}_{n,\alpha}(x;q)$  in  (\ref{L565}). This gives
 \begin{equation}\label{L5065}
\begin{split}
  C_{m}= \frac{q^{m}(q;q)_{m}}{(q^{\alpha+1};q)_{m}(1-q)^{1+\alpha}\Gamma_{q}(\alpha+1)}\int_{0}^{\infty}\frac{x^{\alpha}}{(-x;q)_{\infty}}L_{m}^{\alpha}(x;q) B^{(1)}_{n,\alpha}(x;q) dx.
  \end{split}
\end{equation}
Since  $\{L_{m}^{\alpha}(x;q)\}_{n\in\mathbb{N}}$ is  an orthogonal polynomials sequence then $ C_{m}=0$ for $ m>n$, and
\begin{equation*}
B^{(1)}_{n,\alpha}(x;q) =\sum_{m=0}^{n}C_{m}L_{m}^{\alpha}(x;q).
\end{equation*}
Now, we calculate $C_{m}$. Using (\ref{q19}) in (\ref{L5065}) gives
\begin{equation*}
\begin{split}
  C_{m}=\frac{q^{m}(q;q)_{m}}{(q^{\alpha+1};q)_{m}(1-q)^{1+\alpha}\Gamma_{q}(\alpha+1)}\sum_{k=0}^{n}\left[
                                    \begin{array}{c}
                                      n \\
                                      k \\
                                    \end{array}
                                  \right]_{q}
  \beta_{n-k,\alpha} (q) \int_{0}^{\infty}
  \,\frac{x^{\alpha+k}}{(-x;q)_{\infty}}L_{m}^{\alpha}(x;q)dx.\end{split}
\end{equation*}
Since
\begin{equation*}
\int_{0}^{\infty}\frac{x^{\alpha}}{(-x;q)_{\infty}} L_{m}^{\alpha}(x;q)x^{k} dx=0, \quad \text{for}\,\, k<m,
\end{equation*}
then
\begin{equation*}
\begin{split}
  C_{m}=\frac{q^{m}(q;q)_{m}}{(q^{\alpha+1};q)_{m}(1-q)^{1+\alpha}\Gamma_{q}(\alpha+1)}\sum_{k=m}^{n}\left[
                                    \begin{array}{c}
                                      n \\
                                      k \\
                                    \end{array}
                                  \right]_{q}
  \beta_{n-k,\alpha} (q) \int_{0}^{\infty}
  \,\frac{x^{\alpha+k}}{(-x;q)_{\infty}}L_{m}^{\alpha}(x;q)   dx. \end{split}
\end{equation*}
From (\ref{L2}),  we get
\begin{equation*}
\begin{split}
  C_{m}&=\frac{q^{m}(1-q)^{m-1}(q^{\alpha+m+1};q)_{\infty}}{(q;q)_{\infty}}\sum_{k=m}^{n}\left[
                                    \begin{array}{c}
                                      n \\
                                      k \\
                                    \end{array}
                                  \right]_{q}
  \beta_{n-k,\alpha} (q) \int_{0}^{\infty}
  \,D_{q}^{m}\left(\frac{x^{\alpha+m}}{(-x;q)_{\infty}}\right)x^{k} dx, \end{split}
\end{equation*}
then applying the $q$-integration by part introduced in Lemma \ref{tt} $m$ times, we obtain
\begin{equation*}
\begin{split}
  C_{m}=&\frac{(-1)^{m}q^{m}(1-q)^{m-1}(q^{\alpha+m+1};q)_{\infty}}{(q;q)_{\infty}}\\&\times\sum_{k=m}^{n}\left[
                                    \begin{array}{c}
                                      n \\
                                      k \\
                                    \end{array}
                                  \right]_{q}
  \left(\prod_{i=0}^{m-1}q^{i-k}\right)\frac{[k]_{q}!}{[k-m]_{q}!} \beta_{n-k,\alpha} (q) \int_{0}^{\infty}
 \frac{x^{\alpha+k}}{(-x;q)_{\infty}}dx.
  \end{split}
\end{equation*}

From \cite[Eq. (5.4), P. 465]{mourad},
\begin{equation*}
\frac{1} { \Gamma_{q}(z)}=\frac{sin\pi z}{\pi}\int_{0}^{\infty}
\frac{ t^{-z}}{(-t(1-q);q)_{\infty}}dt, \,\,\,Re\,z>0.
\end{equation*}
Then
\begin{equation*}
  \int_{0}^{\infty}
 \frac{x^{\alpha+k}}{(-x;q)_{\infty}}dx= \frac{\pi}{sin(-\alpha-k)\pi}\frac{1} { \Gamma_{q}(-\alpha-k)}(1-q)^{\alpha+k}.
\end{equation*}
Therefore,
\begin{equation}\label{L700}
\begin{split}
  C_{m}&=\frac{(-1)^{m}q^{m}(q^{\alpha+m+1};q)_{\infty}}{(1-q)^{2}(q,q;q)_{\infty}}\\&\times\sum_{k=m}^{n}
 \left(\prod_{i=0}^{m-1}q^{i-k}\right) \frac{(q;q)_{n}(q^{-\alpha-k};q)_{\infty}}{(q;q)_{n-k}(q;q)_{k-m}}\frac{\pi}{sin(-\alpha-k)\pi}\beta_{n-k,\alpha}(q).
  \end{split}
\end{equation}
Since
\begin{equation}\label{L17}
 \frac{\pi}{sin(-\alpha-k)\pi} =(-1)^{k-1}\frac{\pi}{sin(\alpha\pi)},\quad\prod_{i=0}^{m-1}q^{i-k}=q^{\frac{m(m-1)}{2}}q^{-km},
\end{equation}
then substituting from (\ref{L17}) into (\ref{L700}), we get
{\small\begin{equation*}
\begin{split}
  C_{m}=&\frac{(-1)^{m}q^{m}q^{m(m-1)/2}(q^{\alpha+m+1};q)_{\infty}(q^{-\alpha};q)_{\infty}}{(1-q)^{2}(q,q;q)_{\infty}}\frac{\pi}{sin(\alpha\pi)}\\&\times\sum_{k=m}^{n}
(-1)^{k-1}q^{-km} \frac{(q;q)_{n}(q^{-\alpha-k};q)_{k}}{(q;q)_{n-k}(q;q)_{k-m}}\beta_{n-k,\alpha}(q).
  \end{split}
\end{equation*}}
Using the relation (\ref{f798}), we obtain
\begin{equation*}{\small
\begin{split}
  C_{m}=&\frac{-q^{m}(q^{\alpha+m+1};q^{-\alpha};q)_{\infty}}{(1-q)^{2}(q,q;q)_{\infty}}\frac{\pi}{sin(\alpha\pi)}\\&\times\sum_{k=m}^{n}q^{\frac{k(2n-k+1)}{2}}\frac{ (q^{-n};q)_{k}  (q^{-k};q)_{m}(q^{-\alpha-k};q)_{k}}{(q;q)_{k}}
 \beta_{n-k,\alpha}(q),
\end{split}}
\end{equation*}
and this completes the proof of the theorem.
}
\end{proof}
\vskip0.2cm
%%%%%%%%%%%%%%%%%%%%%%%%%%%%%%%%%%%%%%%%%%%%%%%%%%%%%%%%%%%%%%%%%%%%%%%%%%%%%%%%%%%%%%%%%%%%%%%%%%%%%%%%%%%%%%%%%

{\small The little $q$-Legendre polynomials $( P_{n}(x|q))_{n}$ are defined  by
\begin{equation*}
\begin{split}
 P_{n}(x|q)&=\,_{2}\varphi_{1} \left(
                                                                               \begin{array}{c}
                                                                                 q^{-n},q^{n+1} \\
                                                                                 q \\
                                                                               \end{array}
                                                                           q;qx  \right)\\&=\sum_{k=0}^{n}\frac{(q^{-n};q)_{k}(q^{n+1};q)_{k}}{(q;q)_{k}}\frac{q^{k}x^{k}}{(q;q)_{k}}.
                                                                           \end{split}
\end{equation*}
They satisfy the Rodrigues formula
\begin{equation}\label{L209}
  P_{n}(x|q)=\frac{q^{n(n-1)/2}(1-q)^{n}}{(q;q)_{n}}\,D_{q^{-1}}^{n}(x^{n}(qx;q)_{n}), \quad for\,\, n\geq 0,
\end{equation}
and the orthogonality relation
\begin{equation}\label{L353}
  \int_{0}^{1}P_{m}(x|q)P_{n}(x|q)d_{q}x
  =\frac{(1-q)}{(1-q^{2n+1})}\delta_{mn},\quad for\,\, m,\,n\geq 0,
\end{equation}}
 see \cite{Asky}. {\small Let $\mathbb{P}_{n}=\{g(x):\text {deg}\, g(x)\leq n\}$ with the inner product
\begin{equation*}
  \langle g(x),p(x)\rangle= \int_{0}^{1}g(x)p(x)d_{q}x,
\end{equation*}}
where $p(x),\,g(x)\in\mathbb{P}_{n}$.
\begin{theorem}\label{gg}{\small
Let $g(x)$ $\in\mathbb{P}_{n}$.  Then $g(x)$ can be represented by
\begin{equation*}
g(x)=\sum_{k=0}^{n}C_{k}P_{k}(x|q),
\end{equation*}
where
\begin{equation*}
\begin{split}
  C_{k}=\frac{q^{k(k-1)/2}(1-q)^{k-1}(1-q^{2k+1})}{(q;q)_{k}}\int_{0}^{1}\,D_{q^{-1}}^{k}(x^{k}(qx;q)_{k})g(x)d_{q}x.
  \end{split}
\end{equation*}}
\end{theorem}
\begin{proof}{\small
Since
\begin{equation*}
g(x)=\sum_{k=0}^{n}C_{k}P_{k}(x|q),
\end{equation*}
then by the orthogonality relation (\ref{L353}), we obtain
\begin{equation}\label{L94}
\begin{split}
C_{k}=\frac{(1-q^{2k+1})}{(1-q)} \langle g(x),P_{k}(x|q)\rangle=\frac{(1-q^{2k+1})}{(1-q)}\int_{0}^{1}P_{k}(x|q)g(x)d_{q}x.
\end{split}
\end{equation}
By using (\ref{L209}), we get
\begin{equation*}
\begin{split}
  C_{k}=\frac{q^{k(k-1)/2}(1-q)^{k-1}(1-q^{2k+1})}{(q;q)_{k}}\int_{0}^{1}\,D_{q^{-1}}^{k}(x^{k}(qx;q)_{k})g(x)d_{q}x,
  \end{split}
\end{equation*}
which readily gives the result.}
\end{proof}
\begin{theorem}\label{uu}{\small
For $n\in\mathbb{N}$ and  $x\in\mathbb{C}$,
\begin{align}\label{L175}
B^{(1)}_{n,\alpha}(x;q)\nonumber& =\sum_{k=0}^{n}\lambda_{k}\left(
\sum_{m=k}^{n}(-1)^{m}q^{\frac{m(2n-m+1)}{2}}\,\frac{(q^{-n};q)_{m}(q^{-m};q)_{k}}{(q;q)_{m+k+1}}\beta_{n-m,\alpha}(q)\right)P_{k}(x|q),\\
B^{(2)}_{n,\alpha}(x;q)\nonumber&=\sum_{k=0}^{n}\lambda_{k}\left(\sum_{m=k}^{n}(-1)^{m}q^{nm}\,
\frac{(q^{-n};q)_{m}(q^{-m};q)_{k}}{(q;q)_{m+k+1}}\beta_{n-m,\alpha}(q)\right)P_{k}(x|q),\\
 B^{(3)}_{n,\alpha}(x;q)\nonumber&=\sum_{k=0}^{n}\lambda_{k}\left(\sum_{m=k}^{n}(-1)^{m}q^{\frac{m(4n-m+1)}{4}}\,
\frac{(q^{-n};q)_{m}(q^{-m};q)_{k}}{(q;q)_{m+k+1}}\beta^{(3)}_{n-m,\alpha}(q)\right)P_{k}(x|q),
\end{align}
where
\begin{equation*}
  \lambda_{k}=q^{\frac{-k(k-3)}{2}}(1-q^{2k+1}).
\end{equation*}}
\end{theorem}
\begin{proof}{\small

Substitute with $g(x)=B^{(1)}_{n,\alpha}(x;q)$  in (\ref{L94}), we obtain
\begin{equation}\label{L30211}
\begin{split}
 C_{k} = \dfrac{(1-q^{2k+1})}{(1-q)}\int_{0}^{1}P_{k}(x|q) B^{(1)}_{n,\alpha}(x;q)d_{q}x.
                                  \end{split}
\end{equation}
Since the polynomials $\{P_{k}(x|q)\}$ are orthogonal, then $ C_{k}=0$  for  $k>n$,  and
\begin{equation}\label{L1298}
B^{(1)}_{n,\alpha}(x;q) =\sum_{k=0}^{n}C_{k}P_{k}(x|q).
\end{equation}
Set
\begin{equation*}
B^{(1)}_{n,\alpha}(x;q)=\sum_{m=0}^{n}\left[
                                    \begin{array}{c}
                                      n \\
                                      m \\
                                    \end{array}
                                  \right]_{q}
   \beta_{n-m,\alpha} (q) x^{m}.
\end{equation*}
From (\ref{L30211}),
\begin{equation*}
\begin{split}
 C_{k}&=\dfrac{(1-q^{2k+1})}{(1-q)} \sum_{m=0}^{n}\left[
                                    \begin{array}{c}
                                      n \\
                                      m \\
                                    \end{array}
                                  \right]_{q} \beta_{n-m,\alpha} (q)\int_{0}^{1}P_{k}(x|q) x^{m}d_{q}x,\\&=\frac{(1-q^{2k+1})}{(1-q)} \sum_{m=k}^{n}\left[
                                    \begin{array}{c}
                                      n \\
                                      m \\
                                    \end{array}
                                  \right]_{q} \beta_{n-m,\alpha} (q)\int_{0}^{1}P_{k}(x|q) x^{m}d_{q}x,
                                  \end{split}
\end{equation*}
since
\begin{equation*}
\int_{0}^{1} P_{k}(x|q) x^{m}d_{q}x=0\quad for\,\,m<k.
\end{equation*}
Hence, by the Rodrigues formula  in (\ref{L209}), we obtain
\begin{equation}\label{L3211}
\begin{split}
 C_{k} = \dfrac{(1-q^{2k+1})q^{k(k-1)/2}(1-q)^{k-1}}{(q;q)_{k}} \sum_{m=k}^{n}\left[
                                    \begin{array}{c}
                                      n \\
                                      m \\
                                    \end{array}
                                  \right]_{q} \beta_{n-m,\alpha} (q)\int_{0}^{1}\,D_{q^{-1}}^{k}(x^{k}(qx;q)_{k}) x^{m}d_{q}x.
                                  \end{split}
\end{equation}
Using the $q^{-1}$-integration by parts
\begin{equation}\label{L490}
\int_{0}^{a}f(\frac{t}{q})D_{q^{-1}}g(t)d_{q}t =q\left((fg)(\frac{a}{q})-(fg)(0) \right)-\int_{0}^{a}g(t)D_{q^{-1}}f(t)d_{q}t,
\end{equation}
where $f$ and $g$ are continuous functions at zero, see \cite{Annaby}. This gives

{\small\begin{equation}\label{L403}
\begin{split}
\int_{0}^{1}\,D_{q^{-1}}^{k}(x^{k}(qx;q)_{k}) x^{m}d_{q}x=&q\left[x^{m}D_{q^{-1}}^{k-1}(x^{k}(qx;q)_{k})\right]_{0}^{\frac{1}{q}} \\&-[m]_{q}q^{1-m}\int_{0}^{1}x^{m-1}D_{q^{-1}}^{k-1}(x^{k}(qx;q)_{k})d_{q}x.
\end{split}
\end{equation}}
The first term  on the right hand side of (\ref{L403}) vanishes because
\begin{equation*}
  D_{q^{-1}}(x^{k}(qx;q)_{k})=[k]_{q^{-1}}x^{k-1}(x;q)_{k}+x^{k}D_{q^{-1}}(qx;q)_{k},
\end{equation*}
and
\begin{equation*}
 D_{q^{-1}}^{j}(qx;q)_{k}|_{x=\frac{1}{q}}=a^{k} \frac{[k]_{q}!}{[k-j]_{q}!}(1;q)_{k-j}=0,\quad \text {for}\,\,j=0,1,...k-1.
\end{equation*}
Therefore,
\begin{equation}\label{L1039}
\begin{split}
\int_{0}^{1}\,D_{q^{-1}}^{k}(x^{k}(qx;q)_{k}) x^{m}d_{q}x=-[m]_{q}q^{1-m}\int_{0}^{1}x^{m-1}D_{q^{-1}}^{k-1}(x^{k}(qx;q)_{k})d_{q}x.
\end{split}
\end{equation}
Now, applying (\ref{L490}) $k-1$ times on the right hand side of (\ref{L1039}), and using that  $D_{q^{-1}}^{m}(x^{k}(qx;q)_{k}= 0$  at  $x=0,\, x=\frac{1}{q}$ $(m=0,1,\ldots ,k-1)$ yields
\begin{equation*}
\begin{split}
\int_{0}^{1}\,D_{q^{-1}}^{k}(x^{k}(qx;q)_{k}) x^{m}d_{q}x=(-1)^{k} \left(\prod_{j=0}^{k-1}q^{1-m-j}\right)\frac{[m]_{q}!}{[m-k]_{q}!} \int_{0}^{1}
 x^{m}(qx;q)_{k}\,d_{q}x.
\end{split}
\end{equation*}
Since
\begin{equation*}
\begin{split}
  B_{q}(x,y)=\int_{0}^{1} t^{x-1}(qt;q)_{y-1}d_{q}t=\int_{0}^{1}t^{x-1}\frac{(tq;q)_{\infty}}{(tq^{y};q)_{\infty}}d_{q}t,\,\,\,Re\,(x)>0,\,Re\,(y)>0,\end{split}
\end{equation*}
see \cite[Eq. (1.58), P. 22]{Annaby}, then
{\small\begin{equation}\label{L40653}
\begin{split}
\int_{0}^{1}D_{q^{-1}}^{k}(x^{k}(qx;q)_{k}) x^{m}d_{q}x&=(-1)^{k}q^{\frac {-k(k-1)}{2}+k}q^{-mk}\frac{[m]_{q}!}{[m-k]_{q}!}B_{q}(m+1,k+1)\\&=
(-1)^{k}q^{\frac {-k(k-3)}{2}}q^{-mk}\frac{[m]_{q}!\Gamma_{q}(m+1)\Gamma_{q}(k+1)}{[m-k]_{q}!\Gamma_{q}(m+k+2)}\\&=
(-1)^{k}q^{\frac {-k(k-3)}{2}}q^{-mk}\frac{([m]_{q}!)^{2}[k]_{q}!}{[m-k]_{q}![m+k+1]_{q}!}.
\end{split}
\end{equation}}
Substituting from (\ref{L40653}) into (\ref{L3211}) yields
\begin{equation}\label{L1201}
\begin{split}
 C_{k} &=(-1)^{k}q^{k}(1-q^{2k+1}) \sum_{m=k}^{n}q^{-mk}\frac{(q;q)_{n}(q;q)_{m}}{(q;q)_{n-m}(q;q)_{m-k}(q;q)_{m+k+1}}\beta_{n-m,\alpha}(q)\\&=
 q^{\frac{-k(k-3)}{2}}(1-q^{2k+1})
\sum_{m=k}^{n}(-1)^{m}q^{\frac {m(2n-m+1)}{2}+1}\,\,
\frac{(q^{-n};q)_{m}(q^{-m};q)_{k}}{(q;q)_{m+k+1}}\beta_{n-m,\alpha}(q),
  \end{split}
\end{equation}
where we used the identity in (\ref{f798}). Therefore, from (\ref{L1201}) and (\ref{L1298}), we get the required result for $B^{(1)}_{n,\alpha}(x;q)$. Similarly, we can prove the result for $B^{(k)}_{n,\alpha}(x;q)\,(k=2,3)$. }
\end{proof}

\bibliographystyle{plain}

\bibliography{bib2}

\end{document}